\documentclass[preprint]{article}

\usepackage[margin=1in]{geometry}
\usepackage[utf8]{inputenc}
\usepackage[english]{babel}
\usepackage{amsfonts}
\usepackage{amsmath}
\usepackage{graphicx}
\usepackage{array}
\usepackage{amsthm}
\usepackage{tikz}

\title{Penalization for non-linear hyperbolic system}

\author{Thomas Auphan}

\date\today

\begin{document}
\newtheorem{definition}{Definition}[section]
\newtheorem{theorem}{Theorem}[section]
\newtheorem{proposition}{Proposition}[section]
\newtheorem{lemma}{Lemma}[section]

\renewcommand{\b}[1]{\mathbf{#1}}
\renewcommand{\u}[1]{\underline{#1}}

\maketitle 
\begin{center}
tauphan@cmi.univ-mrs.fr\\
Aix Marseille Universit\'e, CNRS, Centrale Marseille, LATP, UMR 7353, 13453 Marseille France
\end{center}

\begin{abstract}

This paper proposes a volumetric penalty method to simulate the boundary conditions for a non-linear hyperbolic problem. The boundary conditions are assumed to be maximally strictly dissipative on a non-characteristic boundary. This penalization appears to be quite natural since, after a natural change of variable, the penalty matrix is an orthogonal projector. We prove the convergence towards the solution of the wished hyperbolic problem and that this convergence is sharp in the sense that it does not generate any boundary layer, at any order. The proof involves an approximation by asymptotic expansion and energy estimates in anisotropic Sobolev spaces. 
%
\end{abstract}

\let\thefootnote\relax


\section{Introduction}

Non-linear hyperbolic conservation laws models are very common in fluid mechanics, for example let us just cite Euler, MHD and shallow water equations. The physical domain of the fluid is sometimes quite complex and this can be the source of difficulties to provide an efficient numerical scheme. 
Usually, the boundary conditions need a special treatment with a body-fitted mesh for the implementation in simulation codes. Beside the mesh of the scheme often has to be fitted to the shape of the domain. Penalization, such as other immersed boundary methods, can lead to a simpler treatment of the boundary condition and allows one to use fast numerical solver, such as pseudo-spectral solver for instance, see \cite{Jau12,Kol09}. The solution of the original $\b u$ is approximated by the solution of the penalized problem $\b u_{\varepsilon}$, where $\varepsilon \ll 1$ is the penalization parameter. Thus, the error $\|\b u_{\varepsilon} - \b u\|_{H^s}$ has to be controlled. In the optimal case, $\|\b u_{\varepsilon} - \b u\|_{H^s} = \mathcal{O}(\varepsilon)$ when $\varepsilon$ tends to $0$. In the non optimal case, the penalty methods generates boundary layers which ensures a connection between the physical domain and the penalized area. This boundary layer aggravates the convergence rate, even in some cases the $H^1$ penalization error may increases when $\varepsilon$ tends to $0$ because of the generation of oscillations \cite{Ang12,Liu07}.

Immersed boundary methods have been first implemented by Peskin for the numerical simulations of the flow around heart valves \cite{Pes72}. Some examples of application of penalization method are also given by fish-like swimming simulations, see for instance \cite{Ber11}. Error analysis of penalization method for incompressible viscous flow equations, using a BKW method has been performed by Carbou and Fabrie \cite{Car03}. For the wave equation in the one dimensional case, Paccou \emph{et al.} provides a 
theoretical and numerical study of a $L^2$ penalization for a Dirichlet boundary condition \cite{Pac05}.
Penalty method has already been proposed in the semi-linear characteristic case by Fornet and Gu\`es \cite{For09}. The main result of this paper is a penalization technique for a quasilinear hyperbolic problem which does not generate any boundary layer.
To simplify some parts of the proofs, the notation $\partial_0=\partial_t$ has been sometimes used in this paper.
 
\section{Main result}\label{sect_main_result}
In order to avoid issues of compatibility of the initial condition and to focus on the penalization's problem, we consider a boundary value problem instead of an initial boundary value problem. For the same technical reason, we suppose that the solution is null in the past, \emph{i.e.} for $t<0$.

We could consider the non-linear hyperbolic boundary-value problem presented below: 
\begin{equation}
\displaystyle
\left\{\begin{array}{ll} 
\partial_t \b u(t,\b x) + \sum_{j=1}^{d}{ \bar{\b A}_j(\b u(t,\b x)) \partial_j \b u(t,\b x)} = \bar{\b f}(t,\b x, \b u(t,\b x)) & (t,\b x) \in ]-T_0,T[ \times \mathbb{R}^d_+ \\
\b \Theta(\b u(t,\b x',0))=\b 0 & (t, \b x') \in ]-T_0,T[ \times \mathbb{R}^{d-1} \textsf{, \emph{i.e.} } x_d=0\\
\b u_{|t<0} = \b 0 & 
\end{array}
\right.
\end{equation}
But this form does not take into account of some parameters related to $t,\b x$ but not to $\b u$ such as, for instance, the refraction index, the viscosity... In order to have a more general problem which can be applied to a wide range of physical models, let us add a function $\b a : ]-T_0,T[ \times \mathbb{R}^d \to \mathbb{R}^{N'}$ which is supposed to include all this type of information.

Finally, let us consider a hyperbolic boundary-value problem of the form:
\begin{equation}\label{Original_problem}
\displaystyle
\left\{\begin{array}{ll} 
\partial_t \b u(t,\b x) + \sum_{j=1}^{d}{ \bar{\b A}_j(\b a (t,\b x),\b u(t,\b x)) \partial_j \b u(t,\b x)} = \bar{\b f}(\b a(t,\b x),\b u(t,\b x)) & (t,\b x) \in ]-T_0,T[ \times \mathbb{R}^d_+ \\
\b \Theta(\b a (t,\b x',0), \b u(t,\b x',0))=\b 0 & (t,\b x') \in ]-T_0,T[ \times \mathbb{R}^{d-1} \textsf{, \emph{i.e.} } x_d=0\\
\b u_{|t<0} = \b 0 & 
\end{array}
\right.
\end{equation}
In this paper, the space variable writes $\b x=(x_1,\dots,x_d)=(\b x',x_d)$. The space domain is represented in the figure \ref{Fig_domain}.
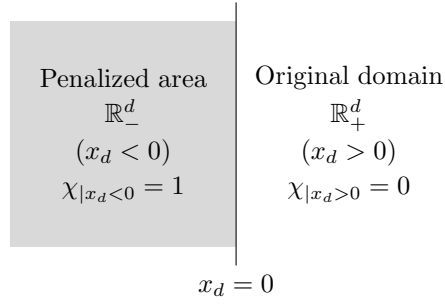
\begin{figure}
\begin{center}
\begin{tikzpicture}
\fill [color=gray!30] (-3,-1.5) rectangle (0,1.5);
\draw (-1.5,0.75) node{Penalized area};
\draw (-1.5,0.25) node{$\mathbb{R}^d_-$};
\draw (-1.5,-0.25) node{$(x_d<0)$};
\draw (-1.5,-0.75) node{$\chi_{|x_d<0}=1$};
\draw (1.5,0.75) node{Original domain};
\draw (1.5,0.25) node{$\mathbb{R}^d_+$};
\draw (1.5,-0.25) node{$(x_d>0)$};
\draw (1.5,-0.75) node{$\chi_{|x_d>0}=0$};
\draw (0,-1.75) -- (0,1.75);
\draw (0,-1.75) node[below]{$x_d=0$};
\end{tikzpicture}
\caption{A schematic representation of the space domain}
\label{Fig_domain}
\end{center}
\end{figure}

We make the following assumptions about the hyperbolic problem (\ref{Original_problem}):
\begin{enumerate}
 \item $\b a : ]-T_0,T[ \times \mathbb{R}^d \to \mathbb{R}^{N'}$ is in $H^{\infty}(]-T_0,T[ \times \mathbb{R}^d)$.
 \item $\bar{\b f} : \mathbb{R}^{N'} \times \mathbb{R}^N \to \mathbb{R}^N$ is $\mathcal{C}^\infty$ and, for all $\b y \in \mathbb{R}^{N'}, \bar{\b f}(\b y, \b 0)= \b 0$.
 \item $\b \Theta : \mathbb{R}^{N'} \times \mathbb{R}^N \to \mathbb{R}^p$ is $\mathcal{C}^\infty$ and for all $(\b y,\b U) \in \mathbb{R}^{N'} \times \mathbb{R}^N,\nabla_{\b u} \b \Theta\left(\b y, \b U\right)$ has a constant rank $p$. Besides, for all $\b y \in \mathbb{R}^{N'}, \b \Theta(\b y,\b 0)=\b 0$.
 \item For every $j$, $\bar{\b A}_j :\mathbb{R}^{N'} \times \mathbb{R}^N \to \mathcal{M}_N(\mathbb{R})$ is $\mathcal{C}^\infty$.
 \item There exists a symmetrizer $\b S(\b y, \b U)$ such that, for all $(\b y, \b U) \in \mathbb{R}^{N'} \times \mathbb{R}^N$:
 \begin{itemize}
  \item $\b S(\b y, \b U)$ is symmetric and positive definite, uniformly in $(\b y, \b U)$ when $\b U$ is in a neighbourhood $\mathcal{U} \subset \mathbb{R}^N$ of $\b 0$ and $\b y$ in a neighbourhood $\mathcal{Z} \subset \mathbb{R}^{N'}$ of $\b 0$. This means that there exists $\bar{e}>0$ such that, for all $(\b y, \b U) \in \mathcal{Z} \times \mathcal{U}$, and for all $\b W \in \mathbb{R}^N$, $\langle \b S(\b y, \b U) \b W, \b W \rangle \geq \bar{e} \|\b W\|^2$, where $\langle , \rangle$ and $\|.\|$ are respectively the euclidean scalar product and norm on $\mathbb{R}^N$.
  \item For all $j \in \{1,\dots,d\}$, $\b S(\b y, \b U) \bar{\b A}_j(\b y, \b U)$ is symmetric.
 \end{itemize} 
\end{enumerate}

We assume that the problem is non characteristic, \emph{i.e.} for all $(\b y, \b U) \in \mathbb{R}^{N'} \times \mathbb{R}^N$ such that $\b \Theta(\b y,\b U)=\b 0$, the matrix $\bar{\b A}_d(\b y, \b U)$ is invertible.
The boundary conditions are assumed to be maximally strictly dissipative: For all $\b y \in \mathcal{Z}$, if there exists $\b U \in \mathbb{R}^N$ such that  $\b \Theta(\b y, \b U)=\b 0$, the quadratic form have the following properties:
\begin{itemize}
 \item $\exists \bar{\mu}>0, \forall \b y \in \mathbb{R}^{N'}, \forall \b W \in \ker \nabla_{\b u} \b \Theta(\b y,\b 0), \langle \b S(\b y, \b U)\bar{\b A}_d(\b y, \b U) \b W, \b W \rangle \leq -\bar{\mu} \| \b W \|^2$.
 \item $\dim \ker \nabla_{\b u} \b \Theta(\b y,\b 0)$ is maximal for the property above.
\end{itemize}

According to \cite{Gue90,Rau74}, one can assert there exists a finite time $\theta>0$ such that the original problem (\emph{cf.} equation (\ref{Original_problem})) admits a unique solution $\b u$ in $H^{\infty}(]-T_0,\b \theta[\times \mathbb{R}^d_+)$.

\begin{lemma}\label{lemma_ch_unknown}
There exists $\mathcal{Q}  \subset \mathcal{U}, \mathcal{V}$, two neighbourhoods of $\b 0 \in \mathbb{R}^N$ and  $\mathcal{Y} \subset \mathcal{Z}$ a neighbourhood of $\b 0 \in \mathbb{R}^{N'}$ satisfying: there exists $\b H \in \mathcal{C}^{\infty}\left(\mathcal{Y}\times \mathcal{V},\mathcal{Q}\right)$ such that, for all $\b y \in \mathcal{Y}$, $\b H(\b y, . )$ is a $\mathcal{C}^{\infty}$-diffeomorphism from $\mathcal{V}$ to $\mathcal{Q}$ and such that
\begin{equation*}
\begin{array}{l}
 \forall \b U \in \mathcal{Q}, \forall \b y \in \mathcal{Y}, \b \Theta(\b y,\b U)=\b 0 \Longleftrightarrow V_1 = V_2 = \dots = V_p = 0\\
 \text{and } \forall \b y \in \mathcal{Y}, \b H(\b y,\b 0)=\b 0
\end{array}
\end{equation*}
Where $\b V \in \mathbb{R}^N$ is such that $\b U=\b H(\b y, \b V)$ and $(V_1,\dots,V_N)=\b V$.
\end{lemma}
\begin{proof}[Proof of the lemma \ref{lemma_ch_unknown}:] The matrix $\nabla_{\b u} \b \Theta(\b 0, \b 0)$ has rank $p$. Eventually re-arranging the terms, let us assume that the square matrix of size $p$ chose columns are $\partial_{u_i} \b \Theta(\b 0, \b 0)$ ($i \in \{1,\dots,p\}$) is invertible.

Let us define the function $\b Z:  (\b U,\b y) \mapsto (\b \Theta(\b y,\b U),U_{p+1}, \dots, U_N,\b y)$ and write $\b V=(\b \Theta(\b y,\b U),U_{p+1}, \dots, U_N)$.

Observe that $\nabla_{\b u,\b a} \b Z(\b 0,\b 0)$ is invertible. The inverse function theorem proves the existence of the neighbourhoods $\mathcal{Q} \subset \mathcal{U} \subset \mathbb{R}^N$ and  $\mathcal{Y} \subset \mathcal{Z} \subset \mathbb{R}^{N'}$ such that $\b Z$ is a $\mathcal{C}^{\infty}$-diffeomorphism defined on $ \mathcal{Q} \times \mathcal{Y}$. The $N$ first components of $\b Z^{-1}$ generates the change of unknown function $\b H$.
\end{proof}

Henceforth, the function $\b a$ is assumed to be valued in the neighbourhood $\mathcal{Y}$.
The proof of the lemma \ref{lemma_ch_unknown} contains a simple choice for the change of unknown $\b H$.

In order to simplify the notations, the dependence of the functions and matrices on $(t,\b x)$ and $\b a(t,\b x)$ is now implicit. So, for instance, $\bar{\b A}_j(\b u)$ stands for $(t,\b x) \mapsto \bar{\b A}_j(\b a(t,\b x),\b u(t,\b x))$ and $\partial_j \left(\bar{\b A}_j(\b u)\right)$ means $\nabla_{\b a} \bar{\b A}_j(\b a(t,\b x),\b u(t,\b x)) \cdot\partial_j \b a(t, \b x)  +  \nabla_{\b u} \bar{\b A}_j(\b a(t,\b x),\b u(t,\b x)) \cdot \partial_j \b u(t, \b x)$.

$\b P$ is defined as the matrix of the projection on the linear subspace $\mathbb{R}^p \times \{0\}^{N-p}$ written in the canonical basis. The boundary condition with the new variables becomes $\b P \b v = \b 0$. For the new unknown $\b v$, the system writes (the parameter function $\b a$ is understood):
\begin{equation} \label{Change_unknown_pb}
\displaystyle
\left\{\begin{array}{ll} 
 \nabla_{\b v} \b H(\b v) \, \partial_t \b v + \sum_{j=1}^{d}{\bar{\b A}_j\left(\b H(\b v)\right) \nabla_{\b v} \b H(\b v) \partial_j \b v } = \bar{\b f}\left(\b H(\b v)\right) & \textsf{ in } ]-T_0,T[ \times \mathbb{R}^d_+\\
 \b P \b v_{|x_d=0}=\b 0 & \textsf{ in } ]-T_0,T[ \times \mathbb{R}^{d-1}
 \end{array}
\right.
\end{equation}
The system is then multiplied on the left by $\nabla_{\b v} \b H(\b v)^{\top} \b S\left(\b H(\b v) \right)$ to obtain:
\begin{equation}\label{Reformulated_pb}
\displaystyle
\left\{\begin{array}{ll} 
 \b A_0(\b v) \, \partial_t \b v + \sum_{j=1}^{d}{\b A_j(\b v) \partial_j \b v } = \b f(\b v) & \textsf{ in } ]-T_0,T[ \times \mathbb{R}^d_+\\
 \b P \b v_{|x_d=0}=\b 0 & \textsf{ in } ]-T_0,T[ \times \mathbb{R}^{d-1}\\
  \b v_{|t<0} = \b 0 & \text{ in } ]-T_0,0[ \times \mathbb{R}^d_+
 \end{array}
\right.
\end{equation}
In this new formulation, the functions $\b A_j$ and $\b f$ are:
\begin{align*}
\b A_0(\b v) &= \nabla_{\b v} \b H(\b v)^{\top} \b S\left(\b H(\b v)\right) \nabla_{\b v} \b H(\b v)\\
\b A_j(\b v) &= \nabla_{\b v} \b H(\b v)^{\top} \b S\left(\b H(\b v)\right) \bar{\b A}_j(\b v) \nabla_{\b v} \b H(\b v)\\
\b f(\b v) &= \nabla_{\b v} \b H(\b v)^{\top} \b S\left(\b H(\b v)\right) \left(\bar{\b f}\left(\b H(\b v)\right)-\nabla_{\b a} \b H(\b v) \cdot \partial_t \b a - \sum_{j=1}^{d}{\bar{\b A}_j(\b v) \nabla_{\b a} \b H(\b v) \cdot \partial_j \b a} \right)
\end{align*}
According to the properties on $\b S \left(\b H(\b v) \right)$ and $\nabla_{\b v} \b H(\b v)$, we can assert that $\b A_0(\b y,\b V)$ is uniformly positive definite regarding $(\b y, \b V)$, where $\b y \in \mathcal{Y}$ and $\b V$ such that $\b H(\b y, \b V) \in \mathcal{Q}$. Hence,
there exists $e>0$ (independent of $\b V$) such that, for all $\b y$ and for all $\b W \in \mathbb{R}^N$, $\langle \b A_0(\b y, \b V) \b W, \b W \rangle \geq e \|\b W\|^2$.

The next lemma recalls a classical invariance property (that can be easily checked):
\begin{lemma}
If the original problem (\ref{Original_problem}) has maximally strictly dissipative boundary conditions, the reformulated problem (\ref{Reformulated_pb}) has also maximal strictly dissipative boundary conditions.
\end{lemma}

For the reformulated problem (\ref{Reformulated_pb}), the property of maximally strictly dissipative boundary conditions means: For all $\b V \in \mathbb{R}^N$ such that $\b P \b V=\b 0$, the quadratic form have the following properties:
\begin{itemize}
 \item $\exists \mu>0, \forall \b W \in \ker \b P, \forall \b y \in \mathcal{Y}, \langle \b A_d(\b y, \b V) \b W, \b W\rangle \leq -\mu \| \b W \|^2$
 \item $N-p$ is the number of strictly negative eigenvalues of $\b A_d(\b y, \b V)$ with multiplicity. Thus, with multiplicity, there are $p$ strictly positive eigenvalues.
\end{itemize}

Let us now introduce the following penalized system, which is the main concern of the paper
\begin{equation}\label{Penalized_syst}
\displaystyle
\left\{\begin{array}{ll}
 \b A_0(\b v_{\varepsilon}) \, \partial_t \b v_{\varepsilon} + \sum_{j=1}^{d}{\b A_j(\b v_{\varepsilon}) \partial_j \b v_{\varepsilon} } + \dfrac{\chi}{\varepsilon} \b P \b v_{\varepsilon} = \b f(\b v_{\varepsilon}) 	& \textsf{ in } ]-T_0,T[ \times \mathbb{R}^d\\
 \b v_{\varepsilon \, | t<0}=\b 0				& \textsf{ in } ]-T_0,0[ \times \mathbb{R}^d
\end{array}\right.
\end{equation}
where $\chi$ is the characteristic function of the obstacle, \emph{i.e} $\chi_{|x_d\leq 0}=1$ and $\chi_{|x_d> 0}=0$, see the figure \ref{Fig_domain}.

%
%
Notice that the boundary condition of the reformulated problem (\ref{Reformulated_pb}) is  $\b P \b v_{|x_d=0}=\b 0$ and the penalization term added in the penalized system (\ref{Penalized_syst}) simply writes $\dfrac{\chi}{\varepsilon} \b P \b v_{\varepsilon}$. Thus, when $\varepsilon$ tends to $0$, from the formal point of view, one recovers the boundary condition $\b P \b v_{\varepsilon | x_d=0} \approx \b 0$.
The main result of this paper is Theorem \ref{Th_reform} (see below) which ensure that the penalized system (\ref{Penalized_syst}) is well-posed and provides an estimation of the error due to the penalization.

\begin{theorem}\label{Th_reform}
Under the assumptions presented above, there exists a finite time $T \in ]0, \theta[$ and $\varepsilon_0>0$ such that, for all $\varepsilon \in ]0,\varepsilon_0]$, the penalized problem 
\begin{equation}
\displaystyle
\left\{\begin{array}{ll}
 \b A_0(\b v_{\varepsilon}) \, \partial_t \b v_{\varepsilon} + \sum_{j=1}^{d}{\b A_j(\b v_{\varepsilon}) \partial_j \b v_{\varepsilon} } + \dfrac{\chi}{\varepsilon} \b P \b v_{\varepsilon} = \b f(\b v_{\varepsilon}) & \textsf{ in } ]-T_0,T[ \times \mathbb{R}^d\\
 \b v_{\varepsilon | t<0} = \b 0 & 
\end{array}\right.
\end{equation}
has a unique solution $\b v_{\varepsilon} \in H^1(]-T_0,T[\times \mathbb{R}^d) \cap W^{1,\infty}(]-T_0,T[\times \mathbb{R}^d)$. Besides, $\b v_{\varepsilon}$ is smooth on each side of the interface $x_d=0$, \emph{i.e.}, $\b v_{\varepsilon | x_d>0} \in H^{\infty}(]-T_0,T[\times \mathbb{R}^d_+)$ and $\b v_{\varepsilon | x_d<0} \in H^{\infty}(]-T_0,T[\times \mathbb{R}^d_-)$.

Moreover, for all $s \in \mathbb{N}$, the following estimate holds as $\varepsilon$ goes to $0$:
\begin{equation*}
\|\b v - \b v_{\varepsilon}\|_{H^s(]-T_0,T[\times \mathbb{R}^d_+)} = \mathcal{O}(\varepsilon)
\end{equation*}

\end{theorem}

Theorem \ref{Th_reform} provides a linear penalization for the reformulated problem. For the original problem (\ref{Original_problem}), the penalization becomes non linear.
Finally, for the hyperbolic problem in the original form, the theorem reads:

\begin{theorem}
Considering the assumption described above for the original problem 
\begin{equation}
\displaystyle
\left\{\begin{array}{ll} 
\partial_t \b u + \sum_{j=1}^{d}{ \bar{\b A}_j(\b u) \partial_j \b u} = \bar{\b f}(\b u) & \textsf{ in } ]-T_0,T[ \times \mathbb{R}^d_+ \\
\b \Theta(\b u)=\b 0 & x_d=0\\
\b u_{|t<0} = \b 0
\end{array}
\right.
\end{equation}
there exists a finite time $T\leq \theta$ and $\varepsilon_0>0$ such that, for all $\varepsilon \in ]0,\varepsilon_0]$, the penalised problem 
\begin{equation*}
\displaystyle
\left\{\begin{array}{ll} 
\begin{array}{l} \partial_t \b u_{\varepsilon} + \sum_{j=1}^{d}{ \bar{\b A}_j(\b u_{\varepsilon}) \partial_j \b u_{\varepsilon}} 
 + \dfrac{\chi(\b x)}{\varepsilon} \b M(\b u_{\varepsilon}) \, \b u_{\varepsilon} = \bar{\b f}(\b u_{\varepsilon})\end{array} & (t,\b x) \in ]-T_0,T[ \times \mathbb{R}^d \\
 \b u_{\varepsilon | t<0} = \b 0 & \\
\end{array}
\right.
\end{equation*} 
has a unique solution $\b u_{\varepsilon} \in H^1(]-T_0,T[\times \mathbb{R}^d_+) \cap  W^{1,\infty}(]-T_0,T[\times \mathbb{R}^d)$ which is smooth on each side of the boundary $x_d=0$: $\b u_{\varepsilon | x_d>0} \in H^{\infty}(]-T_0,T[\times \mathbb{R}^d_+) \textsf{ and }\b u_{\varepsilon | x_d<0} \in H^{\infty}(]-T_0,T[\times \mathbb{R}^d_-)$

where:
\begin{equation*}
\b M(\b u_{\varepsilon}) \, \b u_{\varepsilon} = \left(\b S(\b u_{\varepsilon})\right)^{-1} \left(\nabla_{\b v} \b H\left(\b H^{-1}(\b u_{\varepsilon})\right)^{\top}\right)^{-1} \b P \, \b H^{-1}(\b u_{\varepsilon})
\end{equation*}
Recall that, for $\b S, \b H$ and thus for $\b M$, the dependence on the function $\b a$ is implicit.

For all $s \in \mathbb{N}$, the penalization error estimate when $\varepsilon$ tends to $0$ is given by:
\begin{equation*}
\|\b u_{\varepsilon} - \b u\|_{H^s(]-T_0,T[ \times \mathbb{R}^d_+)} = \mathcal{O}(\varepsilon)
\end{equation*}

\end{theorem}

The penalization matrix is non trivial and is of the form $\b M(\b u_{\varepsilon}) \, \b u_{\varepsilon}$, as 
\begin{equation*}
\left(\b S(\b 0)\right)^{-1} \left(\nabla_{\b v} \b H\left(\b H^{-1}(\b 0)\right)^{\top}\right)^{-1} \b P \, \b H^{-1}(\b 0) = \b 0
\end{equation*}
Observe that, if $p<n$, the penalization matrix is not invertible. Besides, remark that it is not always possible to have a human readable expression of $\b M$.

From the practical point of view, it is simpler to consider the reformulated problem (see theorem \ref{Th_reform}), as in this form, the penalization appears very natural.  Furthermore, even in the linear case, the construction of the penalty matrix $\b P$ is much simpler than the one proposed in the paper \cite{For09}.
The estimate $\|\b u_{\varepsilon} - \b u\|_{H^s} = \mathcal{O}(\varepsilon)$, can be interpreted as an absence of boundary layer for the penalty method described in this paper. This feature differs from the results known for quasilinear hyperbolic problem \cite{Khe08}.

To prove Theorem \ref{Th_reform}, we first build an approximate solution $\b v_a$ of the penalized problem (\ref{Reformulated_pb}) using a formal asymptotic expansion, as presented in the section \ref{sect_BKW}. The second step is to prove that the exact solution of (\ref{Reformulated_pb}) writes $\b v_a + \varepsilon \b w$ together with a good control of $\b w$. 
In order to show that $\b w$ remains bounded in a suitable Sobolev space, an iterative scheme is defined generating a sequence $(\b w^k)_{k\in \mathbb{N}}$. In the section \ref{sect_well_posed}, using energy estimates, we justify that $(\b w^k)$ is bounded for the $L^2$ and the $L^{\infty}$ norms and converges towards a function $\b w$.

The solution of the original problem $\b u$ is defined up to the time $\theta$ but, according to our theorem, the solution of the penalized problem might not be defined up to this time. Indeed, in the formal asymptotic expansion we were not able to prove the existence of the expansion in the penalized area up to the time $\theta$. This is in contrast with the case for the semilinear version of the penalization \cite{For09}.
 
In the following sections, we consider that the open subset $\Omega_T=]-T_0,T[ \times \mathbb{R}^d$. $\Omega_T^+=]-T_0,T[ \times \mathbb{R}^d_+$ represents the original domain (\emph{i.e.} the domain of the boundary value problem (\ref{Original_problem})) and $\Omega_T^-=]-T_0,T[ \times \mathbb{R}^d_-$ the penalized area (\emph{i.e.} the fictitious domain). Fornet and Gu\`es \cite{For09} presented a method to extend the results of Theorem \ref{Th_reform} for a more complicated original domain shape.

The sections \ref{sect_example_lin} and \ref{Example_appl} describe in a few lines two examples of application of this penalization method.

\section{The formal asymptotic expansion}\label{sect_BKW}

In order to build an approximate solution we look, at first, for a formal asymptotic expansion of the continuous solution of the form: 
\begin{equation*}
\b v_{\varepsilon}(t,x) \sim \left\{ \begin{array}{l} \sum_{n=0}^{+\infty}{ \varepsilon^n \b V^{n,-}(t,\b x)} \textsf{ if } x_d<0\\
\sum_{n=0}^{+\infty}{ \varepsilon^n \b V^{n,+}(t,\b x)} \textsf{ if } x_d>0 \end{array} \right.
\end{equation*}
where $\b V^{n,-}$ and $\b V^{n,+}$ satisfies the assumptions presented below:
\begin{itemize}
 \item $\b V^{n,-}_{|t<0}=\b 0$.
 \item $\b V^{n,+}_{|t<0}=\b 0$.
 \item For all $t>-T_0, \b V^{n,-}(t,x_1,\dots,x_{d-1},0)=\b V^{n,+}(t,x_1,\dots,x_{d-1},0)$.
\end{itemize}
We will build the $\b V^{n,-}$ and $\b V^{n,+}$ up to any order $n$. $\b V^{n,\pm}$ represents $\b V^{n,+}$ in the area $x_d>0$ and $\b V^{n,-}$ where $x_d<0$. As the series $\sum_{n}{ \varepsilon^n \b V^{n,\pm}(t,\b x)}$ does not converge in general, this is only a formal expansion and we use the character $\sim$ instead of $=$. The meaning of $\sim$ is in the sense of asymptotic expansions.

As, for all $j \in \{0,\dots,d\}$, $\b A_j(\b a (t,\b x), .)$ and $\b f$ are indefinitely differentiable, the following asymptotic expansions hold:
\begin{align*}
\b A_j\left(\b a(t,\b x), \b v_{\varepsilon}(t,\b x)\right)&\sim \sum_{n=0}^{+\infty} \varepsilon^n \b A_j^n\left(\b a(t,\b x), \b V^{0,\pm}(t,\b x),\dots, \b V^{n,\pm}(t,\b x)\right)\\
\b f\left(\b a (t,\b x), \b v_{\varepsilon}(t,\b x)\right)&\sim \sum_{n=0}^{+\infty} \varepsilon^n \b f^n\left(\b a(t,\b x), \b V^{0,\pm}(t,\b x),\dots, \b V^{n,\pm}(t,\b x)\right)\\
\end{align*}
Substituting the expansions in the system (\ref{Penalized_syst}) gives:
\begin{equation}\label{Penalized_expansion_syst}
\displaystyle
\dfrac{\chi}{\varepsilon} \b P \b V^{0,\pm} \!+\! \sum_{n=0}^{+\infty}{\varepsilon^n \! \left(\! \b A_0^0(\b V^{0,\pm}) \partial_t \b V^{n,\pm} \!+\! \sum_{j=1}^{d}{\b A_j^0(\b V^{0,\pm}) \partial_j \b V^{n,\pm} } \! + \! \b F^n(\b V^{0,\pm}\!,\b V^{k,\pm}\!, \partial \b V^{k-1,\pm}\!, 1 \! \leq \! k \! \leq \! n) \! + \! \chi \b P \b V^{n \! + \! 1} \! \right)} \!\! = \! \b 0
\end{equation}
Where $\b F^n(\b V^{0,\pm}\!,\b V^{k,\pm}\!, \partial \b V^{k-1,\pm}\!, 1 \! \leq \! k \! \leq \! n)$ contains all the remaining terms which depend on $\b V^{0,\pm}\!, \partial_j \b V^{0,\pm}\!, \dots,$ $\b V^{n-1,\pm}$, $\partial_j \b V^{n-1,\pm}$, $\b V^{n,\pm}$ but not on $\partial_j \b V^{n,\pm}$. Observe that the function\\ $\b V^{n,\pm} \mapsto \b F^n\left(\b V^{0,\pm}\!,\b V^{k,\pm}\!, \partial \b V^{k-1,\pm}\!, 1 \! \leq \! k \! \leq \! n\right)$ is affine.

\emph{Now, we consider the induction hypothesis:} $(\mathcal{H}^n):$  There exists a $T>0$ independent of $n$ such that for all $k \leq n, \b V^{k,+}$ and $\b V^{k,-}$  are well-defined on $]-T_0,T[ \times \mathbb{R}^d_+$ and $]-T_0,T[ \times \mathbb{R}^d_-$ (respectively). Besides $\b P \b V^{n+1,+}$ is well-defined on $]-T_0,T[\times \mathbb{R}^d_-$.

\emph{Proof of the initial assumption $(\mathcal{H}^0)$, studying the terms in $\varepsilon^0$:}

According to the term in $\varepsilon^{-1}$, we have $\b P \b V^{0,-} = 0$ (for all $x_d<0$).

\underline{For $x_d>0$ ($\chi(\b x)=0$):}

According to term in $\varepsilon^0$ of the equation (\ref{Penalized_expansion_syst}), $\b V^{0,+}$ satisfies the following hyperbolic system:
\begin{equation}\label{syst_order_0_+}
\left\{\begin{array}{l}
\b A_0^0(\b V^{0,+}) \, \partial_t \b V^{0,+} + \sum_{j=1}^{d}{\b A_j^0(\b V^{0,+}) \partial_j \b V^{0,+} } + \b F^0( \b V^{0,+}) = \b 0 \textsf{ in } ]-T_0,T[\times \mathbb{R}^d_+\\
\b P \b V^{0,+}_{\,|x_d=0}=\b P \b V^{0,-}_{\,|x_d=0}\\
\b V^{0,+}_{\,|t \in ]-T_0,0[} = \b 0
\end{array} \right.
\end{equation}
In fact, this hyperbolic system is exactly the boundary value problem (\ref{Reformulated_pb}), so it has maximally strictly dissipative and non characteristic boundary conditions.
So there exists a unique smooth solution, $\b V^{0,+}\in H^{\infty}(]-T_0,\theta[\times \mathbb{R}^d_+)$, of (\ref{syst_order_0_+}). Remark that, finally, $\b V^{0,+}$ equals to $\b v$, the solution of the reformulated hyperbolic problem (\ref{Reformulated_pb}).

\underline{For $x_d<0$ ($\chi(\b x)=1$):}
\begin{equation}\label{syst_order_0_-}
\left\{\begin{array}{l}
\b A_0^0(\b V^{0,-}) \, \partial_t \b V^{0,-} + \sum_{j=1}^{d}{\b A_j^0(\b V^{0,-}) \partial_j \b V^{0,-} } + \b F^0( \b V^{0,-})+\b P \b V^{1,-} = \b 0 \textsf{ in } ]-T_0,T[\times \mathbb{R}^d_-\\
\b V^{0,-}_{\,|x_d=0}=\b V^{0,+}_{\,|x_d=0}\\
\b V^{0,-}_{\,|t \in ]-T_0,0[} = \b 0
\end{array} \right.
\end{equation}
In order to obtain $\b V^{0,-}$, as $\b P \b V^{0,-}=\b 0$ has already been computed, we only need to construct $(\b I-\b P)\b V^{0,-}$ which is solution of:
\begin{equation}\label{syst_order_0_-_proj}
\left\{\begin{array}{l}
(\b I-\b P) \b A_0^0(\b V^{0,-}) (\b I-\b P) \, \partial_t \left( (\b I-\b P)\b V^{0,-}\right) + \sum_{j=1}^{d}{(\b I-\b P)\b A_j^0(\b V^{0,-}) (\b I-\b P) \partial_j \left( (\b I-\b P)\b V^{0,-}\right) } \\
\qquad + (\b I-\b P)\b F^0( \b V^{0,-})= \b 0 \textsf{ in } ]-T_0,T[\times \mathbb{R}^d_-\\
(\b I-\b P)\b V^{0,-}_{\,|x_d=0}=(\b I-\b P)\b V^{0,+}_{\,|x_d=0}\\
(\b I-\b P)\b V^{0,-}_{\,|t \in ]-T_0,0[} = \b 0
\end{array} \right.
\end{equation}
Let us write 
\begin{equation*}
 \left(\begin{array}{c} \b 0\\ \b V^{0,-}_{I\!I}\end{array}\right)=(\b I - \b P) \b V^{0,-} \qquad \text{, } \left(\begin{array}{c} \b 0\\ \b F^0_{I\!I}( \b V^{0,-})\end{array}\right) =  (\b I-\b P)\b F^0( \b V^{0,-})
\end{equation*}
and define the $N-p \times N-p$ matrices $\b A_{j}^{0,I\!I}(\b V^{0,-})$ such that
\begin{equation*}
 (\b I-\b P)\b A_j^0(\b V^{0,-}) (\b I-\b P)=\left(\begin{array}{c|c} \b 0 & \b 0\\ \hline \b 0 & \b A_{j}^{0, I\!I}(\b V^{0,-})\end{array}\right)
\end{equation*}
The problem (\ref{syst_order_0_-_proj}) can now be rewritten as a hyperbolic problem composed of $N-p$ equations (as its $p$ first components are null):
\begin{equation}\label{syst_order_0_-_proj_reduced}
\left\{\begin{array}{l}
\b A_{0}^{0,I\!I}(\b V^{0,-}) \, \partial_t \b V^{0,-}_{I\!I} + \sum_{j=1}^{d}{\b A_{j}^{0,I\!I}(\b V^{0,-}) \partial_j \b V^{0,-}_{I\!I} } + \b F^0_{I\!I}( \b V^{0,-})= \b 0 \textsf{ in } ]-T_0,T[\times \mathbb{R}^d_-\\
\b V^{0,-}_{I\!I\,|x_d=0}=\b V^{0,+}_{I\!I\,|x_d=0}\\
\b V^{0,-}_{I\!I\,|t \in ]-T_0,0[} = \b 0
\end{array} \right.
\end{equation}
The matrix $\b A_{0}^{0,I\!I}(\b V^{0,-})$ is symmetric positive definite, so do $\b A_{d}^{0,I\!I}(\b V^{0,-})$. To prove the well-posedness of the system (\ref{syst_order_0_-_proj_reduced}), we check that the boundary condition is maximally strictly dissipative:
\begin{align*}
\forall \b W_{I\!I} \in \mathbb{R}^{N-p},& \b W = \left(\begin{array}{c} \b 0\\ \b W_{I\!I}\end{array}\right) \in \mathrm{R}(\b I - \b P) = \ker \b P,\\
\langle \b A_{d}^{0,I\!I}(\b V^{0,-}_{\,|x_d=0}) \b W_{I\!I}, \b W_{I\!I}\rangle_{\mathbb{R}^{N-p}} & = \langle \b A_{d}^{0,I\!I}(\b V^{0,+}_{\,|x_d=0}) \b W_{I\!I}, \b W_{I\!I}\rangle_{\mathbb{R}^{N-p}}\\
		&= \langle \b A_{d}^{0}(\b V^{0,+}_{\,|x_d=0}) \underbrace{\b W}_{\in \ker \b P}, \b W\rangle_{\mathbb{R}^{N-p}}\\
		 & \leq - \mu \|\b W\|^2 = - \mu \|\b W_{I\!I}\|^2
\end{align*}
as the reformulated problem has maximally strictly dissipative boundary conditions.
Thus, the matrix $\b A_{d}^{0,I\!I}(\b V^{0,-}_{\,|x_d=0})$ is symmetric negative definite, which shows that $\mathcal{N}=\{\b 0\}\in \mathbb{R}^{N-p}$ is clearly the space of maximal dimension for which $\exists \mu_1>0, \forall \b W_{I\!I} \in \mathcal{N}, \langle -\b A_{d}^{0,I\!I}(\b V^{0,+}_{\,|x_d=0}) \b W_{I\!I}, \b W_{I\!I}\rangle_{\mathbb{R}^{N-p}} \leq - \mu_1 \| \b W_{I\!I}\|^2$.
Hence, the boundary conditions of the system (\ref{syst_order_0_-_proj_reduced}) are maximally strictly dissipative.

Hence, there exists $T\in ]0, \theta]$ such that there is a unique smooth solution, $ \b V^{0,-}_{I\!I}$, of (\ref{syst_order_0_-_proj_reduced}) defined on $]-T_0,T[\times \mathbb{R}^d_-$. Finally $(\b I - \b P) \b V^{0,-}$ and $\b V^{0,-}$ are built up to the time $T$. \emph{A priori}, it could happen that $T<\theta$.

Then, $\b P \b V^{1,-}$ is computed using:
\begin{equation*}
\b P \b V^{1,-}= - \b P \b A_0^0(\b V^{0,-}) \, \partial_t \b V^{0,-} - \sum_{j=1}^{d}{\b P \b A_j^0(\b V^{0,-}) \partial_j \b V^{0,-} } - \b P \b F^0( \b v^{0,-})
\end{equation*}
\emph{Proof of the induction hypothesis, using the terms in $\varepsilon^n$:} We assume that, for all $k\leq n-1$, $\b V^{k,-}, \b V^{k,+}$ and $\b P \b V^{n,-}$ are built.

\underline{For $x_d>0$ ($\chi(\b x)=0$):}
\begin{equation*}\label{syst_order_n_+}
\left\{\begin{array}{l}
\b A_0^0(\b V^{0,+}) \, \partial_t \b V^{n,+} + \sum_{j=1}^{d}{\b A_j^0(\b V^{0,+}) \partial_j \b V^{n,+} } + \b F^n(\b V^{0,+}\!,\b V^{k,+}\!, \partial \b V^{k-1,+}\!, 1 \! \leq \! k \! \leq \! n) =\b 0 \textsf{ in } ]-T_0,T[\times \mathbb{R}^d_+\\
\b P \b V^{n,+}_{\,|x_d=0}=\b P \b V^{n,-}_{\,|x_d=0}\\
\b V^{n,+}_{\,|t \in ]-T_0,0[} = \b 0
\end{array} \right.
\end{equation*}
As $\b F^n(\b V^{0,+}\!,\b V^{k,+}\!, \partial \b V^{k-1,+}\!, 1 \! \leq \! k \! \leq \! n)$ is affine for the variable $\b V^{n,+}$, the system (\ref{syst_order_n_+}) is a linear hyperbolic problem, and the homogeneous boundary condition version have maximally strictly dissipative boundary condition. So, the hyperbolic problem (\ref{syst_order_n_+}) admits a unique smooth solution $\b V^{n,+}$ up to the time $T$ introduced in the proof of $(\mathcal{H}^0)$ \cite{Ben07, Cha82}.

\underline{For $x_d<0$ ($\chi(\b x)=1$):}

Considering the terms at the order $n$ of (\ref{Penalized_expansion_syst}), $\b V^{n,-}$ satisfies:
\begin{equation}\label{syst_order_n_-}
\left\{\!\!\!\!\begin{array}{l}
\displaystyle \b A_0^0(\b V^{0,-}) \partial_t \b V^{n,-} \! + \! \sum_{j=1}^{d}{\b A_j^0(\b V^{0,-}) \partial_j \b V^{n,-} } \! + \! \b F^n(\b V^{0,-}\!,\b V^{k,-}\!, \partial \b V^{k-1,-}\!, 1 \! \leq \! k \! \leq \! n) \! + \! \b P \b V^{n+1,-} \! = \! \b 0 \textsf{ in } ]\! - \! T_0,T[\times \mathbb{R}^d_-\\
\b V^{n,-}_{\,|x_d=0}=\b V^{n,+}_{\,|x_d=0}\\
\b V^{n,-}_{\,|t \in ]-T_0,0[} = \b 0
\end{array} \right.
\end{equation}
Again, we only need to evaluate $(\b I-\b P)\b V^{n,-}$ to obtain $\b V^{n,-}$. So, we consider the $N-p$ last components, of the following linear system:
\begin{equation*}\label{syst_order_n_-_proj}
\left\{\begin{array}{l}
(\b I-\b P)\b A_0^0(\b V^{0,-})(\b I-\b P) \, \partial_t \left( (\b I-\b P)\b V^{n,-}\right) + \sum_{j=1}^{d}{(\b I-\b P)\b A_j^0(\b V^{0,-})(\b I-\b P) \partial_j \left( (\b I-\b P)\b V^{n,-}\right) } \\ 
\qquad + (\b I-\b P) \b F^n(\b V^{0,-}\!, \b V^{k,-}\!, \partial \b V^{k-1,-}\!, 1 \! \leq \! k \! \leq \! n) = \b 0 \textsf{ in } ]-T_0,T[\times \mathbb{R}^d_-\\
(\b I-\b P)\b V^{n,-}_{\,|x_d=0}=(\b I-\b P)\b V^{n,+}_{\,|x_d=0}\\
(\b I-\b P)\b V^{n,-}_{\,|t \in ]-T_0,0[} = \b 0
\end{array} \right.
\end{equation*}
As it has been done for the case $n=0$ (order $\varepsilon^0$) and $x_d<0$, the solution $(\b I-\b P)\b V^{n,-}$ is finally built up to the time $T$ defined in the proof of $(\mathcal{H}^0)$.

Now, we use the other part of the problem (\ref{syst_order_n_-}), \emph{i.e.} the $p$ first components, to have $\b P \b V^{n+1,-}$.

So $\mathcal{H}^{n}$ is proven and the asymptotic expansion can be built at any order.

The first term of the asymptotic expansion $\b V^{0,\pm}$ is the exact solution of the limit problem, when $\varepsilon$ tends to $0$. As the penalization is incomplete (\emph{i.e.}, the penalization matrix is not invertible), it is necessary to solve a hyperbolic problem in the penalized area ($x_d<0$) to compute $(\b I - \b P) \b V^{0,-}$.

Observe that, to build this asymptotic expansion up to any order, we do not need to introduce any variable of the form $x_d/\varepsilon^b$ (with $b \neq 0$). This is not the case in the paper \cite{For09} (theorem 2.6) where the asymptotic expansion terms are in $\b V^n(t,\b x, x_d/\varepsilon)$. A boundary layer due to a $L^2$ penalty method has also been exhibited thanks to a BKW asymptotic expansion in a paper of Carbou \cite{Car03} for some Brinkmann-type penalization model for viscous flows.
The boundary layer ensures a continuous connection when the conditions at the boundary of the original domain (here, $\mathbb{R}^d_+$) and of the penalized domain are not compatible, which is not the case in our approach. 

\section{Well-posedness and penalization error estimate}\label{sect_well_posed}
The asymptotic expansion built in the previous section may not be the solution of the penalized problem (\ref{Penalized_syst}), it is only a formal expression. But, the first terms (up to an order $M$) will be useful to find the solution of (\ref{Penalized_syst}).

\subsection{Definitions and notations}

We recall the penalized hyperbolic problem considered:
\begin{equation}\left\{\begin{array}{l}
 A_0(\b v_{\varepsilon})\partial_t \b v_{\varepsilon} + \sum_{j=1}^{d} \b A_j(\b v_{\varepsilon}) \partial_j \b v_{\varepsilon} + \frac{1}{\varepsilon} \chi \b P \b v_{\varepsilon} = \b f \quad (t,\b x) \in ]- T_0,T[ \times \mathbb{R}^d\\
 \b v_{\varepsilon |t<0}=\b 0
 \end{array}\right.
 \label{Syst_non_lin_pen}
\end{equation}

In the previous section, we have built an approximate solution $\b v_a(t,\b x)=\sum_{n=0}^{M}{\varepsilon^n \b V^{n,\pm}(t,\b x)}$ (with $M$ large enough) such that
\begin{equation}
\label{Syst_non_lin_pen_app}
\left\{\begin{array}{l}
\b A_0 (\b v_a) \partial_t \b v_a + \sum_{j=1}^{d} \b A_j(\b v_a) \partial_j \b v_a + \frac{1}{\varepsilon} \chi \b P \b v_a = \varepsilon^{M} \b R_{\varepsilon} + \b f \quad (t,\b x) \in ]- T_0,T[ \times \mathbb{R}^d\\
\b v_{a \, |t<0}=\b 0 
 \end{array}\right.
\end{equation}

The proof of Theorem \ref{Th_reform} uses the $L^\infty$ norm of $\b R_{\varepsilon}$ which is bounded independently from $\varepsilon$. This this the object of the lemma below:
\begin{lemma}\label{lemma_R_e}
For some $M \in \mathbb{N}^*$ and $\varepsilon_0>0$, the function $\b v_a=\sum_{n=0}^{M}{\varepsilon^n \b V^{n,\pm}}$ is a solution of this approximate problem
\begin{equation*}
\left\{\begin{array}{l}
\b A_0 (\b v_a) \partial_t \b v_a + \sum_{j=1}^{d} \b A_j(\b v_a) \partial_j \b v_a + \frac{1}{\varepsilon} \chi \b P \b v_a = \varepsilon^{M} \b R_{\varepsilon} + \b f \quad (t,\b x) \in ]- T_0,T[ \times \mathbb{R}^d\\
\b v_{a \, |t<0}=\b 0 
 \end{array}\right.
\end{equation*}
and $\|\b R_{\varepsilon}\|_{\infty}$ is bounded uniformly in $\varepsilon \in ]0,\varepsilon_0]$.
\end{lemma}
\begin{proof}[Proof of the lemma \ref{lemma_R_e}:]
As the asymptotic expansion has a finite order, we can consider:
\begin{align*}
\b A_j\left(\b a(t,\b x), \b v_{a}(t,\b x)\right)& = \sum_{n=0}^{M} \varepsilon^n \b A_j^n\left(\b a(t,\b x), \b V^{0,\pm}(t,\b x),\dots, \b V^{n,\pm}(t,\b x)\right)\\
\b f\left(\b a(t,\b x), \b v_{a}(t,\b x)\right)& = \sum_{n=0}^{M} \varepsilon^n \b f^n\left(\b a(t,\b x), \b V^{0,\pm}(t,\b x),\dots, \b V^{n,\pm}(t,\b x)\right)
\end{align*}
The corrective term $\b R_{\varepsilon}$ satisfies:
\begin{align*}
\varepsilon^{M} \b R_{\varepsilon} & = \b A_0 (\b v_a) \partial_t \b v_a + \sum_{j=1}^{d} \b A_j(\b v_a) \partial_j \b v_a + \frac{1}{\varepsilon} \chi \b P \b v_a -\b f\\
 & = \sum_{j=0}^{d}{\sum_{n=0}^{M} \varepsilon^n \b A_j^n\left(\b V^{0,\pm},\dots, \b V^{n,\pm}\right)\sum_{p=0}^{M}{\varepsilon^p \partial_j \b V^{n,\pm}}} + \sum_{n=0}^{M}\varepsilon^{n-1} \chi \b P \b V^{n,\pm} - \sum_{n=0}^{M} \varepsilon^n \b f^n\left(\b V^{0,\pm},\dots, \b V^{n,\pm}\right)
\end{align*}
According to the definition of the terms $\b V^{n,\pm}$, the equation above reads:
\begin{align*}
\varepsilon^{M} \b R_{\varepsilon} & = \sum_{j=0}^{d}{\sum_{n=M+1}^{2M} \varepsilon^n \sum_{p=n-M}^{M} \b A_j^p\left(\b V^{0,\pm},\dots, \b V^{p,\pm}\right)\partial_j \b V^{n-p,\pm}} - \varepsilon^{M} \chi \b P \b V^{M+1,\pm}
\end{align*}
This is a sum of terms in $H^{\infty}(\Omega_T)$. It follows that there exists a constant $c>0$, independent of $\varepsilon \in ]0,\varepsilon_0]$, such that:
\begin{equation*}
\|\b R_{\varepsilon}\|_{\infty} \leq c \varepsilon^M
\end{equation*}
\end{proof}
For the rest of the proof of Theorem \ref{Th_reform}, we choose $m \geq m_0=\lfloor\frac{d}{2}\rfloor+2$ and $M>3+\frac12 m_0$.

\begin{definition}[Tangential derivatives] Consider $\alpha = (\alpha_0, \dots, \alpha_{d-1}) \in \mathbb{N}^d$, the tangential derivatives operator $\mathcal{T}^{\alpha}$ is defined by $\mathcal{T}^{\alpha}=\partial^{\alpha_0}_{t} \partial^{\alpha_1}_{x_1} \dots \partial^{\alpha_{d-1}}_{x_{d-1}}$. 
\end{definition}

We also define an \emph{ad-hoc} functional space:
\begin{definition}\label{A-space}
We define the space $\mathcal{A}(\Omega_T)$ which is the set of functions $\b v : \Omega_T \to \mathbb{R}^N$ such that:
\begin{itemize}
 \item $\b v \in H^1(\Omega_T)$.
 \item $\b v \in H^{\infty}_{tan}(\Omega_T)$, \emph{i.e.} for all $\alpha \in \mathbb{N}^d, \mathcal{T}^{\alpha} \b v \in L^2(\Omega)$.
 \item $\partial_d \b v \in H^{\infty}_{tan}(\Omega_T)$.
 \item $\b v \in W^{1,\infty}$ which means $\b v \in L^{\infty}$, and, for all $j \in \{0,\dots,d\}, \partial_j \b v \in L^{\infty}$
\end{itemize}
\end{definition}

Now, the objective is to find $\b w \in \mathcal{A}(\Omega_T)$ such as $\b v_{\varepsilon} = \b v_a + \varepsilon \b w$ is a solution of the penalised problem (\ref{Penalized_syst}), \emph{i.e.}:
\begin{equation}
 \label{Syst_non_lin_pen_th}
 \left\{\begin{array}{l}
  \b A_0 (\b v_a + \varepsilon \b w) \partial_t (\b v_a + \varepsilon \b w) + \sum_{j=1}^{d} \b A_j(\b v_a + \varepsilon \b w) \partial_j (\b v_a + \varepsilon \b w) + \frac{1}{\varepsilon} \chi \b P (\b v_a + \varepsilon \b w) = \b f \quad (t,\b x) \in ]- T_0,T[ \times \mathbb{R}^d\\
  \b v_{a \, |t<0} + \varepsilon \b w_{|t<0} =\b 0
 \end{array}\right.
\end{equation}

Taking the difference between (\ref{Syst_non_lin_pen_th}) and (\ref{Syst_non_lin_pen_app}):
\begin{equation*}
\left\{\begin{array}{l}
\sum_{j=0}^{d} \left( \b A_j(\b v_a + \varepsilon \b w) \partial_j (\b v_a + \varepsilon \b w) - \b A_j(\b v_a) \partial_j \b v_a \right) + \frac{1}{\varepsilon} \chi \b P \varepsilon \b w = - \varepsilon^{M} \b R_{\varepsilon} \quad (t,\b x) \in ]- T_0,T[ \times \mathbb{R}^d\\
 \b w_{|t<0}=\b 0 
\end{array}\right.
\end{equation*}

We define the linear map $\b B(\b U, \b V, \varepsilon \b w) : \b W \mapsto \b B(\b U, \b V, \varepsilon \b w) \b W$ such that:
\begin{equation*}
 \sum_{j=0}^{d}{\left(\b A_j(\b v_a + \varepsilon \b w) - \b A_j(\b v_a) \right)\partial_j \b v_a} = - \varepsilon \b B(\b v_a, \nabla \b v_a,\varepsilon \b w) \b w
\end{equation*}
 This operator depends on $(\b a, \b v_a, \nabla \b v_a, \b w)$ but it is of class $\mathcal{C}^{\infty}$ for all its variables. In order to avoid too long equations, we will not write the variables $\b v_a, \nabla \b v_a$ in the operator $\b B(\b v_a, \nabla \b v_a,\varepsilon \b w)$, which now becomes, $\b B(\varepsilon \b w)$.

Hence the hyperbolic problem for $\b w$ reads:
\begin{equation*}\label{Syst_non_lin_pen_w}
\left\{\begin{array}{l}
\b A_0(\b v_a + \varepsilon \b w) \partial_t \b w + \sum_{j=1}^{d}{ \b A_j(\b v_a + \varepsilon \b w) \partial_j \b w } - \b B(\varepsilon \b w) \b w + \frac{1}{\varepsilon} \chi \b P \b w = - \varepsilon^{M-1} \b R_{\varepsilon} \quad (t,\b x) \in ]- T_0,T[ \times \mathbb{R}^d\\
 \b w_{|t<0}=\b 0 
\end{array}\right.
\end{equation*}

We now consider a Picard's iterative scheme:
\begin{align*}\displaystyle
& \b w^0=\b 0\\
& \forall k \in \mathbb{N},\\
& \left\{\begin{array}{l}
\! \! \! \b A_0(\b v_a \! + \! \varepsilon \b w^k) \partial_t \b w^{k+1} \! + \! \sum_{j=1}^{d}{ \b A_j(\b v_a \! + \! \varepsilon \b w^k) \partial_j \b w^{k+1} } \! - \! \b B(\varepsilon \b w^k) \b w^{k+1} \! + \! \dfrac{\chi}{\varepsilon} \b P \b w^{k+1} \! = \! - \varepsilon^{M-1} \b R_{\varepsilon} \textsf{ in } ] \! -  \! T_0,T[ \times \mathbb{R}^d\\
\! \! \! \b w^{k+1}_{|t<0} \! = \! \b 0 
\end{array}\right.
\end{align*}
This sequence is expected to converge towards $\b w$ in $L^2(\Omega_T)$ and then in $H^{\infty}(\Omega_T^+)$, in $H^{\infty}(\Omega_T^-)$ and in $H^{1}(\Omega_T)$.

\subsection{Proof of the convergence of the sequence $(\b w^{k})$ in $L^2$}

\subsubsection{Weighted norms:}
To prove that the sequence $(\b w^{k})$ converges in $L^2(\Omega_T)$, we will use energy estimates. But the $L^2$ norm is not really practical to obtain estimates which are bounded when $\varepsilon$ goes to $0$. That is why the weighted norms presented below are used in this paper:

\begin{definition}[Weighted norms]
\begin{align*}
 \forall \b \Phi \in L^2(\Omega_T), \| \b \Phi \|_{0,\lambda} & = \| e^{- \lambda t} \b \Phi \|_{L^2(\Omega_T)}\\
 \forall \b \Phi \in H^m_{tan}(\Omega_T), \| \b \Phi \|_{m, \lambda} & = \sum_{\vert \alpha \vert \leq m}{\lambda^{m-\vert \alpha \vert}\| \mathcal{T}^{\alpha} \b \Phi \|_{0,\lambda} }\\
 \forall \b \Phi \in H^m_{tan}(\Omega_T), \| \b \Phi \|_{m, \lambda, \varepsilon} & = \sum_{\vert \alpha \vert \leq m}{\lambda^{m-\vert \alpha \vert}\| \sqrt{\varepsilon}^{\vert \alpha \vert} \mathcal{T}^{\alpha} \b \Phi \|_{0,\lambda} }
\end{align*}
\end{definition}

Let us note that:
\begin{enumerate}
 \item $\|.\|_{0,\lambda}$ is equivalent to the norm $\|.\|_{L^2(\Omega_T)}$ (for a fixed value of $\lambda$).
 \item $\|.\|_{m,\lambda}$ and $\|.\|_{m,\lambda,\varepsilon}$ are equivalent to the norm $\|.\|_{H^m_{tan}(\Omega_T)}=\sum_{\vert \alpha \vert \leq m}{\|  \mathcal{T}^{\alpha} . \|_{L^2(\Omega_T)} }$ (for fixed values of $\lambda$ and $\varepsilon$).
 \item We also have ($\b \Phi \in H^m_{tan}(\Omega_T)$):
 \begin{align*}
 \lambda \| \b \Phi \|_{m-1, \lambda, \varepsilon} & \leq \| \b \Phi \|_{m, \lambda, \varepsilon}\\
 \sqrt{\varepsilon} \| \mathcal{T} \b \Phi \|_{m-1, \lambda, \varepsilon} & \leq \| \b \Phi \|_{m, \lambda, \varepsilon}
\end{align*}
\end{enumerate}

It is easy to prove that there exist $c>0$ (depending on $m$, but not on $\lambda, \varepsilon$) and a function $\zeta_m(\lambda)$ (independent of $\varepsilon$) such that
\begin{align*}
 \|\b R_{\varepsilon} \|_{m, \lambda, \varepsilon} &\leq \zeta_m(\lambda)\\
 \|\b R_{\varepsilon}\|_{\infty} &\leq c
\end{align*}
Since $\b R_{\varepsilon |t<0}=\b 0$, observe that $\zeta_m(\lambda)=\mathcal{O}(\lambda^m)$, when $\lambda$ tends to the infinity.

\subsubsection{Energy estimates for $\b w^{k+1}$:}

The goal of this subsection is to provide energy estimates for the following hyperbolic problem of unknown $\tilde{\b w}$: 
\begin{equation}
\label{PDE_energ_estim}
\displaystyle
\left\{\begin{array}{l}
\b A_0(\b v_a + \varepsilon \b b) \partial_t \tilde{\b w} + \sum_{j=1}^{d}{ \b A_j(\b v_a + \varepsilon \b b) \partial_j \tilde{\b w} } - \b B(\varepsilon \b b) \tilde{\b w} + \frac{1}{\varepsilon} \chi \b P \tilde{\b w} = \b g \quad (t,\b x) \in \Omega_T\\
 \tilde{\b w}_{|t<0}=\b 0 
\end{array}\right.
\end{equation}
In this subsection, $\b b$ represents $\b w^k$ (for some $k \in \mathbb{N}$) and $\tilde{\b w}$ stands for $\b w^{k+1}$. In our estimates, the constants must not depend on 
$\b b$, $\tilde{\b w}$ (\emph{i.e.} $\b w^k$, $\b w^{k+1}$) to ensure that we can prove by induction the boundedness of the sequence $(\b w^k)_{k \in \mathbb{N}}$ (for $\|.\|_{m,\lambda, \varepsilon}$ and $\|.\|_{\infty}+\|\nabla .\|_{\infty}$). $\b g$ represents $\varepsilon^{M-1} R_{\varepsilon}$, so we assume $\|\b g\|_{m,\lambda, \varepsilon} \leq \zeta_m(\lambda) \varepsilon^{M-1}$ and $\|\b g\|_{\infty}\leq c \varepsilon^{M-1}$.

\begin{proposition}\label{1st_energy_estim}
We assume that:
\begin{itemize}
 \item $\b b \in \mathcal{A}\left( ]-T_0,T[\times \mathbb{R}^d \right)$.
 \item $\| \b b \|_{\infty} + \| \nabla \b b \|_{\infty} \leq R$.
\end{itemize}
There exists $\varepsilon_1(R)\in]0,1[$  such that the hyperbolic problem (\ref{PDE_energ_estim}) admits a solution $\tilde{\b w} \in \mathcal{A}(]-T_0,T[\times \mathbb{R}^d)$ for all $\varepsilon\leq \varepsilon_1(R)$. Besides, there exists $C(R)$ (which does not depends on $\b g$) and $\lambda_0(R)>0$ such that the following energy estimates holds:
\begin{equation*}
\forall \lambda > \lambda_0(R), \sqrt{\lambda} \| \tilde{\b w} \|_{0,\lambda} + \frac{1}{\sqrt{\varepsilon}} \|\chi \b P \tilde{\b w}\|_{0,\lambda} \leq \frac{C(R)}{\sqrt{\lambda}} \|\b g\|_{0,\lambda}
\end{equation*}
\end{proposition}

\begin{proof}[Proof of the proposition \ref{1st_energy_estim}:]
We use the following notations:
\begin{equation*}
 \b v_{a-}(t,x_1,\dots, x_d)= \b v_{a}(t,x_1,\dots, -x_d) \textsf{ and } \b v_{a+}(t,x_1,\dots, x_d)= \b v_{a}(t,x_1,\dots, + x_d)
\end{equation*}
And we consider this boundary value problem:
\begin{equation*}
\left\{ \begin{array}{l}
\left(\!\!\! \begin{array}{cc} \b A_0(\b u_{a-}+\varepsilon \b b_{-}) & \b 0\\ \b 0 & \b A_0(\b u_{a-}+\varepsilon \b b_{-}) \end{array}\!\!\!\right) \partial_t \left(\!\!\!\begin{array}{c} \tilde{\b w}_-\\ \tilde{\b w}_+ \end{array}\!\!\!\right) + \displaystyle\sum_{j=1}^{d-1}{\left(\!\!\! \begin{array}{cc} \b A_j(\b u_{a-}+\varepsilon \b b_{-}) & \b 0\\ \b 0 & \b A_j(\b u_{a-}+\varepsilon \b b_{-}) \end{array}\!\!\!\right) \partial_j \left(\!\!\! \begin{array}{c} \tilde{\b w}_-\\ \tilde{\b w}_+ \end{array} \!\!\! \right)   } \\
 \qquad + \underbrace{\left(\!\!\! \begin{array}{cc} - \b A_d(\b u_{a-}+\varepsilon \b b_{-}) & \b 0\\ \b 0 & \b A_d(\b u_{a-}+\varepsilon \b b_{-}) \end{array} \!\!\!\right)}_{=\mathbb{A}_d} \partial_d \left(\!\!\! \begin{array}{c} \tilde{\b w}_-\\ \tilde{\b w}_+ \end{array} \!\!\! \right)
 - \left( \!\!\! \begin{array}{cc} \b B(\varepsilon \b b_-) \tilde{\b w}_- & \b 0 \\ \b 0 & \b B(\varepsilon \b b_+) \tilde{\b w}_+ \end{array} \!\!\! \right) \left( \!\!\! \begin{array}{c} \tilde{\b w}_-\\ \tilde{\b w}_+ \end{array}\!\!\! \right) \\
  \qquad + \frac{1}{\varepsilon}\left(\!\!\! \begin{array}{c} \tilde{\b w}_-\\ \b 0 \end{array}\!\!\! \right) = \left(\!\!\! \begin{array}{c} \b g_-\\ \b g_+ \end{array}\!\!\!\right)\\
 \tilde{\b w}_{-|t<0}=\b 0\\
 \tilde{\b w}_{+|t<0}=\b 0\\
 \tilde{\b w}_{-|x_d=0}-\tilde{\b w}_{+|x_d=0}=\b 0
\end{array}\right.
\end{equation*}
Observe that the hyperbolic problem above is symmetric and has maximally dissipative boundary conditions:
\begin{itemize}
 \item At $x_d = 0$, we have $\b A_d(\b v_{a-} + \varepsilon \b b_-)_{|x_d=0} = \b A_d(\b v_{a+} + \varepsilon \b b_+)_{|x_d=0}$ and $\tilde{\b w}_- = \tilde{\b w_+}$.\\ So, $\langle \mathbb{A}_d \left(\begin{array}{c} \tilde{\b w}_-\\ \tilde{\b w}_+\end{array}\right),  \left(\begin{array}{c} \tilde{\b w}_-\\ \tilde{\b w}_+\end{array}\right) \rangle_{\mathbb{R}^{2N}}=0$.
 \item $\b A_d(\b v_{a+} + \varepsilon \b b_+)$ is symmetric and invertible, for $\varepsilon$ sufficiently small ($\varepsilon < \varepsilon_1(R)$). So the eigenspace associated to the negative eigenvalues of $\mathbb{A}_d$ is of dimension $n$.
\end{itemize}

The results from Rauch \cite{Rau74}, Gu\`es \cite{Gue90}, Benzoni-Serre \cite{Ben07} or Chazarain-Piriou \cite{Cha82} (page 475, theorem 6.10) let us claim the existence and the uniqueness of $(\tilde{\b w}_-,\tilde{\b w}_+) \in H^{\infty}(\Omega_T^+)^2$.
According to Sobolev inclusions, we have also $(\tilde{\b w}_-,\tilde{\b w}_+) \in W^{1,\infty}(\Omega_T^+)^2$.

The solution of (\ref{PDE_energ_estim}) can be written:
\begin{equation*}
\tilde{\b w}(t,x_1,\dots,x_d)=\left\{ \begin{array}{l} \tilde{\b w}_-(t,x_1,\dots,-x_d) \textsf{ if } x_d<0\\ \tilde{\b w}_+(t,x_1,\dots,x_d) \textsf{ if } x_d \geq 0\end{array}\right.
\end{equation*}
Furthermore, $\tilde{\b w}$ and its tangential derivatives $\mathcal{T}^{\alpha} \tilde{\b w}$ are in $L^2(\Omega_T)$.

Computing $\partial_d \tilde{\b w}$ thanks to the equation (\ref{PDE_energ_estim}) let us assert that $\tilde{\b w} \in H^1(\Omega_T) \cap H^\infty_{tan}(\Omega_T)$. Finally, $\tilde{\b w} \in \mathcal{A}(\Omega_T)$.

\textbf{Energy estimates for $\tilde{\b w}$:}
We define $\tilde{\b w}_{\lambda}(t,\b x)=\exp\left(-\lambda t \right) \tilde{\b w}(t,\b x)$, such that $\|\tilde{\b w}_{\lambda}\|_{L^2(\Omega_T)}=\|\tilde{\b w}\|_{0,\lambda}$.

As $\b A_0(\b v_a + \varepsilon \b b)$ is uniformly positive definite ($\langle \b A_0(\varepsilon \b b) \tilde{\b w}_{\lambda}, \tilde{\b w}_{\lambda}\rangle_{L^2(\Omega_T)} \geq e_0 \|\tilde{\b w}_{\lambda}\|_{L^2(\Omega_T)}$):
\begin{equation}\label{PDE_lambda}\displaystyle
\left\{\begin{array}{l}
\lambda \b A_0(\varepsilon \b b)\tilde{\b w}_{\lambda} + \sum_{j=0}^{d}{ \b A_j(\varepsilon \b b) \partial_j \tilde{\b w}_{\lambda} } - \b B(\varepsilon \b b) \tilde{\b w}_{\lambda} + \frac{1}{\varepsilon} \chi \b P \tilde{\b w}_{\lambda} = \b g_{\lambda} \quad (t,\b x) \in ]- T_0,T[ \times \mathbb{R}^d\\
 \tilde{\b w}_{\lambda |t<0}=\b 0 
\end{array}\right.
\end{equation}
The $L^2(\Omega_T)$ inner product of (\ref{PDE_lambda}) with $\tilde{\b w}_{\lambda}$ reads:
\begin{equation}\displaystyle
 \lambda e_0 \|\tilde{\b w}_{\lambda}\|_{L^2(\Omega_T)}^2 + \sum_{j=0}^{d}{ \langle \b A_j(\varepsilon \b b) \partial_j \tilde{\b w}_{\lambda}, \tilde{\b w}_{\lambda} \rangle_{L^2(\Omega_T)}} - \| \b B(\varepsilon \b b)\|_{\infty} \|\tilde{\b w}_{\lambda}\|_{L^2(\Omega_T)}^2 + \frac{1}{\varepsilon} \|\chi \b P \tilde{\b w}_{\lambda}\|_{L^2(\Omega_T)}^2 \leq \|\b g_{\lambda}\|_{L^2(\Omega_T)} \|\tilde{\b w}_{\lambda}\|_{L^2(\Omega_T)}
\end{equation}
 Besides, we have, for $j \in \{1,\dots,d\}$:
 \begin{align*}
  \langle \b A_j(\varepsilon \b b) \partial_j \tilde{\b w}_{\lambda}, \tilde{\b w}_{\lambda} \rangle_{L^2(\Omega_T)} & = - \frac12 \langle \partial_j \left(\b A_j(\varepsilon \b b)\right) \tilde{\b w}_{\lambda}, \tilde{\b w}_{\lambda} \rangle_{L^2(\Omega_T)}\\
  		& \geq - \frac12 \|\partial_j \left(\b A_j(\varepsilon \b b)\right)\|_{\infty} \| \tilde{\b w}_{\lambda} \|_{L^2(\Omega_T)}^2
 \end{align*}
 For the time derivative term:
  \begin{align*}
  \langle \b A_0(\varepsilon \b b) \partial_0 \tilde{\b w}_{\lambda}, \tilde{\b w}_{\lambda} \rangle_{L^2(\Omega_T)} \! & = \! - \frac12 \langle \partial_0 \left(\b A_0(\varepsilon \b b)\right)  \tilde{\b w}_{\lambda}, \tilde{\b w}_{\lambda} \rangle_{L^2(\Omega_T)} \! + \! \frac12 \underbrace{\int_{\mathbb{R}^d}{\!\!\tilde{\b w}_{\lambda | t=T}^\top \b A_0(\varepsilon \b b)_{t=T}\tilde{\b w}_{\lambda | t=T} d\b x}}_{\geq 0} \\ 
  & \qquad - \! \frac12 \underbrace{\int_{\mathbb{R}^d}{\!\!\tilde{\b w}_{\lambda | t=-T_0}^\top \b A_0(\varepsilon \b b)_{t=T_0}\tilde{\b w}_{\lambda | t=-T_0} d\b x}}_{=0 \textsf{ as } \tilde{\b w}_{t<0}=\b 0}\\
  		& \geq - \frac12 \|\partial_j \left(\b A_0(\varepsilon \b b)\right)\|_{\infty} \| \tilde{\b w}_{\lambda} \|_{L^2(\Omega_T)}^2
 \end{align*}
 \begin{equation}\displaystyle
 \lambda e_0 \|\tilde{\b w}_{\lambda}\|_{L^2(\Omega_T)}^2 + \frac{1}{\varepsilon} \|\chi \b P \tilde{\b w}_{\lambda}\|_{L^2(\Omega_T)}^2 \! \leq \! \left(\!\!\frac12 \sum_{j=0}^{d}{ \|\partial_j \left(\b A_j(\varepsilon \b b)\right)\|_{\infty}} \! + \! \|\b B(\varepsilon \b b)\|_{\infty}\!\!\right) \! \|\tilde{\b w}_{\lambda}\|_{L^2(\Omega_T)}^2 + \|\b g\|_{L^2(\Omega_T)} \|\tilde{\b w}_{\lambda}\|_{L^2(\Omega_T)}
\end{equation}
Replacing, in the preceding inequality, the term $\frac{1}{\varepsilon} \|\chi \b P \tilde{\b w}_{\lambda}\|_{L^2(\Omega_T)}^2$ by $0$ leads to:
\begin{equation}\displaystyle
\forall \lambda>0, \lambda e_0 \|\tilde{\b w}_{\lambda}\|_{L^2(\Omega_T)} \leq \left(\frac12 \sum_{j=0}^{d}{ \|\partial_j \left(\b A_j(\varepsilon \b b)\right)\|_{\infty}} + \|\b B(\varepsilon \b b)\|_{\infty}\right) \|\tilde{\b w}_{\lambda}\|_{L^2(\Omega_T)} + \|\b g_{\lambda}\|_{L^2(\Omega_T)}
\end{equation}
Then assuming that $\lambda>\lambda_0(R) = \frac{2}{e_0} \left(\frac12 \sum_{j=0}^{d}{ \|\partial_j \left(\b A_j(\varepsilon \b b)\right)\|_{\infty}} + \|\b B(\varepsilon \b b)\|_{\infty}\right)$
\begin{equation}\displaystyle
\lambda  \|\tilde{\b w}_{\lambda}\|_{L^2(\Omega_T)} \leq \frac{2}{e_0 \lambda_0(R)} \|\b g\|_{L^2(\Omega_T)}
\end{equation}
About the estimate on the term $\|\chi \b P \tilde{\b w}_{\lambda}\|_{L^2(\Omega_T)}$, the same process as above is applied:
\begin{align*}
\frac{1}{\sqrt{\varepsilon}} \|\chi \b P \tilde{\b w}_{\lambda}\|_{L^2(\Omega_T)} & \leq \left(\left(\frac12 \sum_{j=0}^{d}{ \|\partial_j \left(\b A_j(\varepsilon \b b)\right)\|_{\infty}} + \|\b B(\varepsilon \b b)\|_{\infty}\right) \|\tilde{\b w}_{\lambda}\|_{L^2(\Omega_T)}^2 + \|\b g\|_{L^2(\Omega_T)} \|\tilde{\b w}_{\lambda}\|_{L^2(\Omega_T)}\right)^{\dfrac12}\\
		& \leq \frac{C_1(R)}{\sqrt{\lambda}} \|\b g_{\lambda}\|_{L^2(\Omega_T)} \qquad \textsf{ using } \lambda \|\tilde{\b w}_{\lambda}\|_{L^2(\Omega_T)} \leq \frac{2}{e_0} \|\b g\|_{L^2(\Omega_T)} \quad (\lambda>\lambda_0(R))
\end{align*}
Finally, defining $C(R)=\frac{2}{e_0}+C_1(R)$, we obtain the energy estimate:
\begin{equation*}
 \sqrt{\lambda}  \|\tilde{\b w}\|_{0,\lambda}+\frac{1}{\sqrt{\varepsilon}} \|\chi \b P \tilde{\b w}\|_{0,\lambda}\leq \frac{C(R)}{\sqrt{\lambda}}\|\b g\|_{0,\lambda}
\end{equation*}

\end{proof}

\subsubsection{Estimates for the tangential derivatives of $\b w^{k+1}$:}
The goal of this subsection is to extend the estimate of the proposition \ref{1st_energy_estim} to the tangential derivatives of $\b w^{n+1}$.

\begin{proposition}\label{Prop_deriv_energ_estim}
We choose $R>0$, $\b b \in \mathcal{A}(]-T_0,T[\times \mathbb{R}^d)$ and $\tilde{\b w} \in \mathcal{A}(]-T_0,T[\times \mathbb{R}^d)$ the solution of the problem (\ref{PDE_energ_estim}).
The following properties are assumed:
\begin{itemize}
 \item $\|\b b\|_{\infty} + \| \nabla \b b \|_{\infty} \leq R$
 \item $0<\varepsilon<\varepsilon_1(R)$
 \item $\lambda>\lambda_0(R)>1$
\end{itemize}
There exists $Q(R)$ (which does not depends on $\tilde{\b w}, \b b,\lambda,\varepsilon$) such that $\tilde{\b w}$ satisfies the estimate:
\begin{equation}\label{Estim_deriv}
\sqrt{\lambda} \| \tilde{\b w}\|_{m,\lambda,\varepsilon}+\frac{1}{\sqrt{\varepsilon}}\|\chi \b P \tilde{\b w}\|_{m,\lambda,\varepsilon} \leq \frac{Q(R)}{\sqrt{\lambda}} \left(\|\b b\|_{m,\lambda,\varepsilon} \left(\|\mathcal{T} \tilde{\b w}\|_{\infty} + \|\chi \b P \tilde{\b w}\|_{\infty} + \|\b g\|_{\infty}\right) + \|\b g\|_{m,\lambda,\varepsilon} \right)
\end{equation}

\end{proposition}

\begin{proof}[Proof of the proposition \ref{Prop_deriv_energ_estim}:]
Let us note $\alpha \in \mathbb{N}^d$ such that $|\alpha|\leq m$. In this proof, we consider that $cst(R)$ is a constant which does not depends on $\tilde{\b w}, \b b ,\lambda,\varepsilon$. 

For two operators $\b S$ and $\b Q$, the commutation operator $[\b S,\b Q]$ is defined by the formula:
\begin{equation*}
[\b S,\b Q]\b \Phi = \b S\left(\b Q(\b \Phi)\right) - \b Q\left(\b S(\b \Phi \right))
\end{equation*}
After applying the tangential derivative operator $\sqrt{\varepsilon}^{|\alpha|} \mathcal{T}^{\alpha}$ to (\ref{PDE_energ_estim}), the new hyperbolic problem reads:
\begin{equation}\label{Penal_dem_deriv}
\left\{\!\!\begin{array}{l}
  \sum_{j=0}^{d}{\b A_j(\varepsilon \b b) \partial_j \left(\sqrt{\varepsilon}^{|\alpha|} \mathcal{T}^{\alpha} \tilde{\b w} \right)} - \b B(\varepsilon \b b) \left(\sqrt{\varepsilon}^{|\alpha|} \mathcal{T}^{\alpha} \tilde{\b w} \right) + \frac{1}{\varepsilon} \chi \b P \left(\sqrt{\varepsilon}^{|\alpha|} \mathcal{T}^{\alpha} \tilde{\b w} \right) = \\
  \quad \b A_d(\varepsilon \b b) \left(-\displaystyle\sum_{j=0}^{d-1}{\left[\b A_d^{-1}(\varepsilon \b b) \b A_j(\varepsilon \b b) \partial_j, \sqrt{\varepsilon}^{|\alpha|} \mathcal{T}^{\alpha} \right] \tilde{\b w}} - \frac{1}{\varepsilon} \left[\chi \b A_d^{-1}(\varepsilon \b b) \b P, \sqrt{\varepsilon}^{|\alpha|} \mathcal{T}^{\alpha} \right] \tilde{\b w} +  \sqrt{\varepsilon} \mathcal{T}^{\alpha}\left( A_d^{-1}(\varepsilon \b b) \b g\right) \right)\\
 \sqrt{\varepsilon}^{|\alpha|} \mathcal{T}^{\alpha} \tilde{\b w}_{]-T_0,0[}=\b 0
\end{array}\right.
\end{equation}

We will apply the proposition \ref{1st_energy_estim} to the hyperbolic problem (\ref{Penal_dem_deriv}), with the unknown $\sqrt{\varepsilon}^{|\alpha|} \mathcal{T}^{\alpha} \tilde{\b w}$, and so, we need to estimate the $\|.\|_{0,\lambda}$ norm of the right hand side. For the terms composed of a product, the lemma presented below gives a useful inequality (for a proof of this lemma, see \cite{Gue90}):

\begin{lemma}[Gagliardo-Niremberg-Moser inequality]\label{lemma_Gagliardo-Niremberg-Moser}
Let us consider $\b \Phi_1,\dots,\b \Phi_p \in H^m_{tan}(\Omega_T) \cap L^{\infty}(\Omega_T)$, $\b \alpha_{.,1},\dots,\b \alpha_{.,p} \in \mathbb{N}^d$ ($\b \alpha_{.,l}=(\alpha_{0,l},\dots,\alpha_{d,l})^t$) and $k \in \mathbb{N}$, such that $\sum_{l=1}^{p}{\sum_{i=0}^d{\alpha_{i,l}}}\leq k \leq m$. So, there exists $r>0$, independent of $\varepsilon, \lambda, \b \Phi_1,\dots,\b \Phi_p$ such that:
\begin{equation*}
\lambda^{m-k} \sqrt{\varepsilon}^{\sum_{l=1}^{p}{\sum_{i=0}^d{\alpha_{i,l}}}} \| \mathcal{T}^{\b \alpha_{.,1}} \b \Phi_1 \dots \mathcal{T}^{\b \alpha_{.,p}} \b \Phi_p \|_{0,\lambda} \leq  r \, \sum_{l=1}^p{\left(\prod_{q \neq l} \|\b \Phi_q\|_{\infty}\right) \|\b \Phi_l\|_{m,\lambda,\varepsilon}}
\end{equation*}
\end{lemma}

Let us begin with the second term in the right hand side of (\ref{Penal_dem_deriv}), which is the more delicate to estimate because of the $\varepsilon^{-1}$.

It is necessary to provide an estimate for $\left[\chi \b A_d^{-1} \b P, \sqrt{\varepsilon}^{|\alpha|} \mathcal{T}^{\alpha}\right] \tilde{\b w}$. 
First, this term is expanded by this way ($\b L_{\beta_., \gamma_.,\delta}$ are matrices):
\begin{equation}
\left[\chi \b A_d^{-1} \b P, \sqrt{\varepsilon}^{|\alpha|} \mathcal{T}^{\alpha}\right] \tilde{\b w} = \chi \sum_{\begin{array}{l}\sum \beta_p + \sum \gamma_q \\ + \delta \leq \alpha \end{array}}{\b L_{\beta_., \gamma_.,\delta}(\varepsilon \b b) \mathcal{T}^{\beta_1} \b v_a \dots \mathcal{T}^{\beta_k} \b v_a \mathcal{T}^{\gamma_1} (\varepsilon \b b) \dots \mathcal{T}^{\gamma_l} (\varepsilon \b b) \mathcal{T}^{\delta} \b P \tilde{\b w}}
\end{equation}
where $|\delta|<|\alpha|$.
Considering the $\|.\|_{0,\lambda}$ norm of $\left[\chi \b A_d^{-1} \b P, \sqrt{\varepsilon}^{|\alpha|} \mathcal{T}^{\alpha}\right] \tilde{\b w}$ gives the estimate below:
\begin{align*}
\frac{\lambda^{m-|\alpha|}}{\varepsilon} & \left\| \left[\chi \b A_d^{-1} \b P, \sqrt{\varepsilon}^{|\alpha|} \mathcal{T}^{\alpha}\right] \tilde{\b w}\right\|_{0,\lambda}\\
	& = \frac{\lambda^{m-|\alpha|}}{\varepsilon}  \left\| \chi \sqrt{\varepsilon}^{|\alpha|} \sum_{\begin{array}{l}\sum \beta_p + \sum \gamma_q \\ + \delta \leq \alpha \end{array}}{\b L_{\beta_., \gamma_.,\delta}(\varepsilon \b b) \mathcal{T}^{\beta_1} \b v_a \dots \mathcal{T}^{\beta_k} \b v_a \mathcal{T}^{\gamma_1} (\varepsilon \b b) \dots \mathcal{T}^{\gamma_l} (\varepsilon \b b) \mathcal{T}^{\delta} \b P \tilde{\b w}}  \right\|_{0,\lambda}\\
	& \leq \frac{\lambda^{m-|\alpha|}}{\varepsilon} \Bigg( \sum_{(\gamma_1,\dots\,\gamma_l)=0}{ cst(R) \left\|\sqrt{\varepsilon}^{|\alpha|}\chi \mathcal{T}^{\beta_1} \b v_a \dots \mathcal{T}^{\beta_k} \b v_a \mathcal{T}^{\delta} \b P \tilde{\b w} \right\|_{0,\lambda}}\\
	& \qquad + \sum_{(\gamma_1,\dots\,\gamma_l)\neq 0}{ cst(R)  \left\| \sqrt{\varepsilon}^{|\alpha|} \chi \mathcal{T}^{\beta_1} \b v_a \dots \mathcal{T}^{\beta_k} \b v_a \mathcal{T}^{\gamma_1} (\varepsilon \b b) \dots \mathcal{T}^{\gamma_l} (\varepsilon \b b) \mathcal{T}^{\delta} \b P \tilde{\b w} \right\|_{0,\lambda}} \Bigg)\\
	&\leq \Bigg( \sum_{(\gamma_1,\dots\,\gamma_l)=0} {cst(R) \sqrt{\varepsilon}^{|\alpha|-|\delta|} \lambda^{|\delta|-|\alpha|} \lambda^{m-|\delta|}  \|\sqrt{\varepsilon}^{|\delta|} \chi \mathcal{T}^{\delta} \b P \tilde{\b w}\|_{0,\lambda}}\\
	& \qquad + \! \! \!\sum_{(\gamma_1,\dots\,\gamma_l)\neq 0}{\! \! \! \varepsilon cst(R) \left( \|\b b\|_{m,\lambda,\varepsilon} \|\chi \b P \tilde{\b w}\|_{\infty} + \|\chi \b P \tilde{\b w}\|_{m,\lambda,\varepsilon}\right)} \Bigg) \quad \textsf{ (Gagliardo-Niremberg-Moser estimates)}
\end{align*}
Finally, we obtain:
\begin{equation*}
\frac{\lambda^{m-|\alpha|}}{\varepsilon} \left\| \left[\chi \b A_d^{-1} \b P, \sqrt{\varepsilon}^{|\alpha|} \mathcal{T}^{\alpha}\right] \tilde{\b w}\right\|_{0,\lambda} \leq 
 cst(R) \left(\frac{1}{\sqrt{\varepsilon} \lambda}\|\chi \b P \tilde{\b w}\|_{m,\lambda,\varepsilon} + \left( \|\b b\|_{m,\lambda,\varepsilon} \|\chi \b P \tilde{\b w}\|_{\infty} + \|\chi \b P \tilde{\b w}\|_{m,\lambda,\varepsilon}\right) \right)\\ 
\end{equation*}
Let us note that the role of the coefficients $\sqrt{\varepsilon}^{|\alpha|}$ in the definition of the norm $\|.\|_{m,\lambda,\varepsilon}$ appears in this estimate:
indeed, these coefficients avoid the presence of a $\varepsilon^{-1}$, replacing it by a $1/\sqrt{\varepsilon}$.

For the first term of the right hand side of (\ref{Penal_dem_deriv}), the treatment is more classical because it only has derivative orders less or equal to $|\alpha|$. So, using the Gagliardo-Niremberg-Moser inequality (lemma \ref{lemma_Gagliardo-Niremberg-Moser}):
\begin{align*}
\forall j \in \{0, \dots, d-1\}, \lambda^{m-|\alpha|} \left\| \left[ \b A_d^{-1}(\varepsilon \b b) \b A_j(\varepsilon \b b) \partial_j, \sqrt{\varepsilon}^{|\alpha|} \mathcal{T}^{\alpha} \right] \tilde{\b w} \right\|_{0,\lambda} & \leq cst(R) \left( \| \varepsilon \b b \|_{m, \lambda, \varepsilon} \| \mathcal{T} \tilde{\b w} \|_{\infty} + \|\tilde{\b w}\|_{m,\lambda,\varepsilon} \right)\\
		& \leq cst(R) \left( \| \varepsilon \b b \|_{m, \lambda, \varepsilon} \| \mathcal{T} \tilde{\b w} \|_{\infty} + \|\tilde{\b w}\|_{m,\lambda,\varepsilon} \right)\\
\end{align*}

Applying the Gagliardo-Niremberg-Moser estimate to the last term of the right hand side of (\ref{Penal_dem_deriv}), there exists $cst(R)$ which does not depends on $\b b, \b g, \lambda, \varepsilon$ such that: 
\begin{equation*}
\|\sqrt{\varepsilon}^{|\alpha|} \mathcal{T}^{\alpha}\left(\b A_d^{-1}(\varepsilon \b b) \b g \right)\|_{0,\lambda} \leq cst(R) \left(\|\b b\|_{m,\lambda,\varepsilon} \|\b g \|_{\infty} + \|\b g\|_{m,\lambda,\varepsilon} \right)
\end{equation*}

Finally, the combination of the three estimates leads to the result of the proposition \ref{Prop_deriv_energ_estim}.
\end{proof}

\subsubsection{$L^{\infty}$ estimates}
In order to obtain the induction property, it is necessary to prove that $\| \tilde{\b w}\|_{\infty} + \| \nabla \tilde{\b w}\|_{\infty}\leq R$, for some $\lambda$, and some $\varepsilon$ sufficiently small.
\begin{lemma}\label{lemma_l_infty_estimate_w}
Remember that $m_0=\lfloor \frac{d}{2} \rfloor +2$.
Assume that $\| \b g\|_{m,\lambda,\varepsilon}\leq \zeta_m(\lambda)\varepsilon^{M-1}$ (with $\zeta_m(\lambda)>1$), $\| \b g\|_{\infty} \leq c \varepsilon^{M-1}$, $\|\b b \|_{m,\lambda,\varepsilon} \leq \varepsilon^{M-1}$ and $\|\b b \|_{\infty}+\|\nabla \b b \|_{\infty} \leq 1$. 
There exists $\lambda_0(1),\varepsilon_0(\lambda)$ (which does not depend on $\b b,\tilde{\b w}$) such that, for $\lambda\geq \lambda_0(1)$ and $0<\varepsilon\leq \varepsilon_0(\lambda)$ the following estimates holds:

\begin{equation*}
\|\tilde{\b w}\|_{\infty} + \|\nabla \tilde{\b w}\|_{\infty}\leq 1
\end{equation*}
\begin{equation*}
\|\tilde{\b w}\|_{\infty} + \|\mathcal{T} \tilde{\b w}\|_{\infty} \leq \zeta_m(\lambda) \varepsilon^{M-\frac12 m_0-\frac52}
\end{equation*}
\begin{equation*}
\| \tilde{\b w}\|_{m,\lambda,\varepsilon} \leq \zeta_m(\lambda) \varepsilon^{M-1}
\end{equation*}
\begin{equation*}
\|\sqrt{\varepsilon} \partial_d \tilde{\b w}\|_{m-1,\lambda,\varepsilon} \leq D(\lambda) \varepsilon^{M-\frac32}
\end{equation*}
Where $D(\lambda)$ is a positive function which does not depend on $\b b, \tilde{\b w}, \varepsilon$. 
\end{lemma}
\begin{proof}[Proof of the lemma \ref{lemma_l_infty_estimate_w}:] To prove this lemma let us consider the equality:
\begin{equation}\label{equa_partial_d}
\partial_d \tilde{\b w} = \b A_d^{-1}(\varepsilon \b b)\left(- \sum_{j=0}^{d_1}{ \b A_j(\varepsilon \b b) \partial_j \tilde{\b w}} +  \b B(\varepsilon \b b) \tilde{\b w} - \frac{\chi}{\varepsilon} \b P \tilde{\b w} + \b g \right)
\end{equation}
For sake of simplicity, let us consider again, for each inequality, that $cst$ is a constant real value, \emph{i.e.} which does not depends on $\varepsilon, \lambda, \b b, \tilde{\b w}$.
Thanks to the Gagliardo-Niremberg-Moser estimate, we obtain:
\begin{align*}
\|\sqrt{\varepsilon} \partial_d \tilde{\b w}\|_{m-1,\lambda,\varepsilon} & \leq cst \Big(\sqrt{\varepsilon} \left(\|\mathcal{T} \tilde{\b w}\|_{m-1,\lambda,\varepsilon} \! + \! \| \tilde{\b w }\|_{m-1,\lambda,\varepsilon} \! + \! \|\b b\|_{m-1,\lambda,\varepsilon} (\|\mathcal{T} \tilde{\b w}\|_{\infty} \! + \! \|\tilde{\b w}\|_{\infty} \!+ \!\|\b g\|_{\infty}) \! + \!\| \b g \|_{m-1,\lambda,\varepsilon}\right) \\ 
 & \qquad+ \frac{1}{\sqrt{\varepsilon}} \left(\| \chi \b P \tilde{\b w }\|_{m-1,\lambda,\varepsilon} \!+\! \|\b b\|_{m-1,\lambda,\varepsilon} \| \chi \b P \tilde{\b w }\|_{\infty}\right) \Big)
\end{align*}
Recall that $0<\varepsilon\leq \varepsilon_1(1)$ and $\lambda\geq \lambda_0(1)$.
Comparing the norms $\|.\|_{m-1,\lambda,\varepsilon}$ and $\|.\|_{m,\lambda,\varepsilon}$:
\begin{align*}
\|\sqrt{\varepsilon} \partial_d \tilde{\b w}\|_{m-1,\lambda,\varepsilon} & \leq cst \Bigg( \left(1 + \frac{\sqrt{\varepsilon}}{\lambda}\right) \|\tilde{\b w}\|_{m,\lambda,\varepsilon} + \frac{\sqrt{\varepsilon}}{\lambda} \|\b b\|_{m,\lambda,\varepsilon}(\|\mathcal{T} \tilde{\b w}\|_{\infty} + \|\tilde{\b w}\|_{\infty} + \|\b g\|_{\infty}) + \frac{\sqrt{\epsilon}}{\lambda} \|\b g\|_{m,\lambda,\varepsilon} \\ 
& \qquad + \frac{1}{\lambda \sqrt{\varepsilon}} \left(\| \chi \b P \tilde{\b w }\|_{m,\lambda,\varepsilon} + \|\b b\|_{m,\lambda,\varepsilon} \| \chi \b P \tilde{\b w }\|_{\infty}\right)\Bigg)
\end{align*}
Adding the term $\|\tilde{\b w}\|_{m,\lambda,\varepsilon}$ leads to:
\begin{align*}
\|\tilde{\b w}\|_{m,\lambda,\varepsilon} + \|\sqrt{\varepsilon} \partial_d \tilde{\b w}\|_{m-1,\lambda,\varepsilon} & \leq cst \Bigg(  \|\tilde{\b w}\|_{m,\lambda,\varepsilon} \! + \! \frac{1}{\lambda \sqrt{\varepsilon}} \| \chi \b P \tilde{\b w }\|_{m,\lambda,\varepsilon} \! + \! \frac{\sqrt{\varepsilon}}{\lambda} \|\b b\|_{m,\lambda,\varepsilon}(\|\mathcal{T} \tilde{\b w}\|_{\infty} \! + \! \|\tilde{\b w}\|_{\infty} \! + \! \|\b g\|_{\infty})\\
 & \qquad + \frac{1}{\lambda \sqrt{\varepsilon}}\|\b b\|_{m,\lambda,\varepsilon} \| \tilde{\b w }\|_{\infty} \! + \! \zeta_m(\lambda) \frac{\varepsilon^{M-\frac12}}{\lambda}\Bigg)\\
\|\tilde{\b w}\|_{m,\lambda,\varepsilon} + \|\sqrt{\varepsilon} \partial_d \tilde{\b w}\|_{m-1,\lambda,\varepsilon}& \leq cst \Bigg(\frac{1}{\sqrt{\lambda}} \left( \sqrt{\lambda} \|\tilde{\b w}\|_{m,\lambda,\varepsilon} \! + \! \frac{1}{\sqrt{\varepsilon}} \| \chi \b P \tilde{\b w }\|_{m,\lambda,\varepsilon}\right)\\
 & \qquad + \frac{1}{\lambda \sqrt{\varepsilon}}\|\b b \|_{m,\lambda,\varepsilon} \left(\|\mathcal{T} \tilde{\b w}\|_{\infty} \! + \! \|\tilde{\b w}\|_{\infty} \! + \! \|\b g\|_{\infty}\right) +  \zeta_m(\lambda) \frac{\varepsilon^{M-\frac12}}{\lambda}\Bigg)
\end{align*}
According to the proposition \ref{Prop_deriv_energ_estim}, 
\begin{align*}
\|\tilde{\b w}\|_{m,\lambda,\varepsilon} + \|\sqrt{\varepsilon} \partial_d \tilde{\b w}\|_{m-1,\lambda,\varepsilon} & \leq cst \Bigg(\frac{1}{\lambda} \left(\underbrace{\|\b b\|_{m,\lambda,\varepsilon}}_{\leq \zeta_m(\lambda) \varepsilon^{M-1}} \left(\|\mathcal{T}\tilde{\b w} \|_{\infty} + \|\tilde{\b w} \|_{\infty} + \underbrace{\|\b g \|_{\infty}}_{\leq c \varepsilon^{M-1}} \right) + \varepsilon^{M-1}\right) \\
 & \qquad + \frac{1}{\lambda \sqrt{\varepsilon}}\|\b b \|_{m,\lambda,\varepsilon} \left(\|\mathcal{T} \tilde{\b w}\|_{\infty} + \|\tilde{\b w}\|_{\infty} + \|\b g\|_{\infty}\right)+ \zeta_m(\lambda)\frac{\varepsilon^{M-1}}{\lambda} \Bigg)\\
\end{align*} 
Remembering that $\lim_{\lambda \to \infty} \zeta_m(\lambda)= +\infty$ and that $\varepsilon\leq 1$, we define $\xi$ (independent of $\lambda, \varepsilon$), satisfying:
\begin{equation*}
 \xi \geq \frac{1+\frac{\zeta_m(\lambda)}{\lambda}(1+2 c)}{\zeta_m(\lambda)}
 \end{equation*} 
\begin{equation}\label{partial_d_estim}
\|\tilde{\b w}\|_{m,\lambda,\varepsilon} + \|\sqrt{\varepsilon} \partial_d \tilde{\b w}\|_{m-1,\lambda,\varepsilon} \leq \frac{cst}{\lambda} \varepsilon^{M-\frac32} \left( \|\mathcal{T} \tilde{\b w}\|_{\infty} + \|\tilde{\b w} \|_{\infty} + \xi \, \zeta_m(\lambda) \sqrt{\varepsilon} \right)
\end{equation}
Thanks to the Cauchy-Schwartz and Parseval equality, we can prove the lemma below (see \cite{Gue90} for more details):
\begin{lemma}\label{L_infty_estim_lemme}
Let us define $d \in \mathbb{N}^*, m\geq m_0= \lfloor \frac{d}{2} \rfloor + 2, 0<\varepsilon<1$ and $\lambda>0$. There exists a number $\kappa \geq 0$ (which only depends on $d, m_0,T_0$ and $T$) such that for all $\b \Phi \in H^m_{tan}(\Omega_T) \cap L^{\infty}(\Omega_T)$:
\begin{equation*}
\|\b \Phi\|_{\infty} \leq \kappa \frac{e^{\lambda T}}{\lambda^{m-m_0} \sqrt{\varepsilon}^{m_0+1}}\left(\|\b \Phi\|_{m,\lambda, \varepsilon} + \|\sqrt{\varepsilon} \partial_d \b \Phi\|_{m-1,\lambda, \varepsilon} \right)
\end{equation*}
\end{lemma}

Applying the lemma \ref{L_infty_estim_lemme} to $\tilde{\b w}$, leads to:
\begin{equation*}
\|\tilde{\b w}\|_{\infty} \leq cst \frac{e^{\lambda T}}{\lambda^{m-m_0} \sqrt{\varepsilon}^{m_0+1}} \left( \|\tilde{\b w}\|_{m,\lambda,\varepsilon} + \|\sqrt{\varepsilon} \partial_d \tilde{\b w}\|_{m-1,\lambda,\varepsilon} \right)
\end{equation*}
 As $\sqrt{\varepsilon} \| \mathcal{T} \tilde{\b w}\|_{m-1,\lambda,\varepsilon} \leq  \| \tilde{\b w}\|_{m,\lambda,\varepsilon}$ and $\sqrt{\varepsilon} \|\sqrt{\varepsilon} \partial_d \tilde{\b w}\|_{m-2,\lambda,\varepsilon} \leq \|\sqrt{\varepsilon} \partial_d \tilde{\b w}\|_{m-1,\lambda,\varepsilon}$.
\begin{align*} 
\|\mathcal{T} \tilde{\b w}\|_{\infty} & \leq cst \frac{e^{\lambda T}}{\lambda^{m-1-m_0}\sqrt{\varepsilon}^{m_0+1}}\left(\| \mathcal{T} \tilde{\b w}\|_{m-1,\lambda,\varepsilon} + \|\sqrt{\varepsilon} \partial_d \mathcal{T} \tilde{\b w}\|_{m-2,\lambda, \varepsilon} \right)\\
 & \leq cst \frac{e^{\lambda T}}{\lambda^{m-1-m_0}\sqrt{\varepsilon}^{m_0+2}}\left(\| \tilde{\b w}\|_{m,\lambda,\varepsilon} + \|\sqrt{\varepsilon} \partial_d \tilde{\b w}\|_{m-1,\lambda, \varepsilon} \right)
\end{align*}
So, as $\lambda \geq \lambda_0(1)$ and $0<\varepsilon\leq \varepsilon_1(1)$, we obtain:
\begin{align*}
\|\tilde{\b w}\|_{\infty} + \|\mathcal{T} \tilde{\b w}\|_{\infty} & \leq cst \frac{e^{\lambda T}}{\lambda^{m-1-m_0}\sqrt{\varepsilon}^{m_0+2}}\left(\| \tilde{\b w}\|_{m,\lambda,\varepsilon} + \|\sqrt{\varepsilon} \partial_d \tilde{\b w}\|_{m-1,\lambda, \varepsilon} \right)
\end{align*}
Thanks to the inequality (\ref{partial_d_estim}), we have:
\begin{equation*}
\|\tilde{\b w}\|_{\infty} + \|\mathcal{T} \tilde{\b w}\|_{\infty} \leq cst \frac{e^{\lambda T} \varepsilon^{M-\frac32}}{\lambda^{m-m_0}\sqrt{\varepsilon}^{m_0+2}}\left(\|\mathcal{T} \tilde{\b w}\|_{\infty} + \|\tilde{\b w}\|_{\infty} + \xi \, \zeta_m(\lambda) \sqrt{\varepsilon} \right) 
\end{equation*}
Setting $\lambda$ and considering $\varepsilon_2(\lambda) \in ]0,\varepsilon_1(1)]$ such that $cst \frac{e^{\lambda T} }{\lambda^{m-m_0}} \sqrt{\varepsilon_2(\lambda)} \leq \frac12$ leads to:
\begin{equation*}
\|\tilde{\b w}\|_{\infty} + \|\mathcal{T} \tilde{\b w}\|_{\infty} \leq \zeta_m(\lambda) \varepsilon^{M-\frac12 m_0-\frac52}
\end{equation*}
As $M-\frac12 m_0 - \frac52>0$, there exists $\varepsilon_3(\lambda) \in ]0,\varepsilon_2(\lambda)]$ such that for all $\varepsilon \leq \varepsilon_3(\lambda), \|\tilde{\b w}\|_{\infty} + \|\mathcal{T} \tilde{\b w}\|_{\infty} \leq \frac12$.
Using the equality (\ref{equa_partial_d}) with $M>\frac12 m_0 + 3$, we can assert that there exists $\varepsilon_0(\lambda)\in ]0,\varepsilon_3(\lambda)]$ for $\varepsilon$ sufficiently small, $\|\partial_d \tilde{ \b w}\|_{\infty} \leq \frac12$.

Thus, for $\lambda \geq \lambda_0$ and $\varepsilon \leq \varepsilon_0(\lambda)$: $\| \tilde{\b w} \|_{m,\lambda,\varepsilon} \leq \zeta_m(\lambda) \varepsilon^{M-1}$.
\end{proof}
\subsection{End of the proof of Theorem \ref{Th_reform}}

The first term $\b w^0=\b 0$ satisfies the initial assumptions ($\b w^0 \in \mathcal{A}(]-T_0,T[\times\mathbb{R}^d)$) and the estimates necessary to perform the induction. Besides, consider a fixed value of $\lambda \geq \lambda_{0}(1)$ (\emph{i.e.} $R=1$) and $ \varepsilon \in ]0,\varepsilon_0(\lambda)]$, if $\|\b w^k\|_{m,\lambda,\varepsilon}\leq \zeta_m(\lambda) \varepsilon^{M-1}$ (for any $m \geq m_0$), and $\|\b w^k\|_{\infty} + \|\nabla \b w^k\|_{\infty} \leq 1$, the previous results let us assert that:
\begin{align*}
\|\b w^{k+1}\|_{\infty} + \|\nabla \b w^{k+1}\|_{\infty} & \leq 1\\
\| \b w^{k+1}\|_{m,\lambda,\varepsilon} &\leq \zeta_m(\lambda)\varepsilon^{M-1}\\
\|\sqrt{\varepsilon} \partial_d \b w^{k+1}\|_{m-1,\lambda,\varepsilon} & \leq D(\lambda)\varepsilon^{M-1}
\end{align*}
We conclude, by induction, that the sequence $(\b w^k)$ is bounded for the norm $\|.\|_{m,\lambda,\varepsilon}$ and thus for the $L^2$ norm.

In order to obtain the convergence of $(\b w^k)$, we will prove that this is a Cauchy sequence. Let us remind the iterative scheme:
\begin{equation*}
\sum_{j=0}^{d}{\b A_j(\varepsilon \b w^k) \partial_j \b w^{k+1}} - \b B(\varepsilon  \b w^k)\b w^{k+1} + \frac{\chi}{\varepsilon} \b P \b w^{k+1} = - \varepsilon^{M-1} \b R_{\varepsilon}
\end{equation*}
Let us take the difference between the two systems for $\b w^{k+1}$ (see above) and for $\b w^{k+2}$ :
\begin{equation*}
\left(\sum_{j=0}^{d}{\!\! \b A_j(\! \varepsilon \b w^k) \partial_j} - \! \b B(\! \varepsilon  \b w^k)\! + \! \frac{\chi}{\varepsilon} \b P \!\! \right) \!\! \left(\! \b w^{k \! + \! 1} \!\!\!- \!\b w^{k \! + \! 2}\right) \! =\!-\!\sum_{j=0}^{d}{\!\! \left(\! \b A_j(\! \varepsilon \b w^{k \! + \! 1}) \! - \! \b A_j(\! \varepsilon \b w^{k})\right) \! \partial_j \b w^{k \! + \! 2}} \! + \!
\left(\! \b B(\! \varepsilon  \b w^{k \! + \! 1}) \! - \! \b B(\! \varepsilon  \b w^k)\right) \! \b w^{k \! + \! 2}
\end{equation*}
As $\lambda\geq \lambda_0(1)$, the energy estimate of the proposition \ref{1st_energy_estim} gives, for $\varepsilon$ small enough ($cst$ is a constant):
\begin{equation*}
 \|\b w^{k+2}-\b w^{k+1}\|_{0,\lambda}\leq \frac{cst}{\lambda} \left\| \sum_{j=0}^{d}{\!\left(\b A_j(\varepsilon \b w^{k+1}) \! - \! \b A_j(\varepsilon \b w^{k})\right) \partial_j \b w^{k+2}} \! - \! \left(\b B(\varepsilon  \b w^{k+1}) \! - \! \b B(\varepsilon  \b w^k)\right) \! \b w^{k+2} \right\|_{0,\lambda}
\end{equation*}
As the matrices $\b A_j$ and $\b B$ have continuous coefficients regarding the variables $(\b y, \b v)$  and as $\| \b w^{k}\|_{\infty} \leq 1$, $\| \b w^{k+1}\|_{\infty} \leq 1$, $\| \b w^{k+2}\|_{\infty}+ \| \nabla \b w^{k+2}\|_{\infty} \leq 1$, we can prove that, for $\varepsilon$ sufficiently small:
\begin{equation*}
 \|\b w^{k+2}-\b w^{k+1}\|_{0,\lambda}\leq \frac12 \|\b w^{k+1}-\b w^{k}\|_{0,\lambda}
\end{equation*}
So $(\b w^k)$ is a Cauchy sequence for the norm $\|.\|_{0,\lambda}$ (thus also for $\|.\|_{L^2(\Omega_T)}$). Hence, the sequence $(\b w^k)$ converges towards $\b w \in L^2(\Omega_T)$.

To finish the proof of Theorem \ref{Th_reform}, it remains to claim that $\b w \in \mathcal{A}(\Omega_T)$ and that $\b w$ is a solution of the hyperbolic problem \ref{Syst_non_lin_pen_w}.

First, notice that, according to the distributional sense, $\mathcal{T} \b w^k \to \mathcal{T} \b w$ and $\partial_d \b w^k \to \partial_d \b w$.

As $(\b w^k)$ is bounded for the norm $\| . \|_{m,\lambda,\varepsilon}$, this sequence has a subsequence which converges weakly in $H^m_{tan}(\Omega_T)$ and in $H^1(\Omega_T)$ (because $\|\partial_d \b w^k\|_{0,\lambda}$ is also bounded). So $\b w \in H^m_{tan}(\Omega_T) \cap H^1(\Omega_T)$.

Thanks to the Lebesgue's dominated convergence theorem, $\b v_{\varepsilon}=\b v_a + \varepsilon \b w$ is a solution of the penalized hyperbolic problem (\ref{Penalized_syst}). Besides $L^2$ energy estimates let us ensure the uniqueness of $\b v_{\varepsilon}$.

By induction on $p \in \mathbb{N}$, we can prove that, for each $p\leq m$, $\partial_d^p \b v_{\varepsilon | \, x_d>0} \in H^{m-p}_{tan}(]-T_0,T[\times \mathbb{R}^d_+)$ and $\partial_d^p \b v_{\varepsilon | \, x_d<0} \in H^{m-p}_{tan}(]-T_0,T[\times \mathbb{R}^d_-)$.
\begin{itemize}
 \item The case $p=0$ has already been proven. 
 \item Assume that $p\leq m$ and that for all $k \in \{0,\dots, p-1\}$ , $\partial_d^{k} \b v_{\varepsilon}$ is in $H^{m-k}_{tan}(]-T_0,T[\times \mathbb{R}^d_+)$ and in $H^{m-k}_{tan}(]-T_0,T[\times \mathbb{R}^d_-)$.
 We have 
 \begin{equation*}
 \partial_d^{p} \b v_{\varepsilon}=\partial^{p-1}_d\left(\b A_d^{-1}(\b v_{\varepsilon}) \left( -\sum_{j=0}^{d-1}{\b A_j(\b v_{\varepsilon}) \partial_j \b v_{\varepsilon}}-\frac{\chi}{\varepsilon}\b P \b v_{\varepsilon}+\b f(\b v_{\varepsilon})\right)\right)
 \end{equation*}
So, according to the induction hypothesis and the regularity of the coefficients, we can prove that, for any $m \in \mathbb{N}, \partial_d^{p} \b v_{\varepsilon}$ is in $H^{m-p}_{tan}(]-T_0,T[\times \mathbb{R}^d_+)$ and in $H^{m-p}_{tan}(]-T_0,T[\times \mathbb{R}^d_-)$.
\end{itemize}

Finally $\b v_{\varepsilon}$ is in $H^{\infty}(]-T_0,T[\times \mathbb{R}^d_+)$ and in $H^{\infty}(]-T_0,T[\times \mathbb{R}^d_-)$.
So $\b v_{\varepsilon} \in \mathcal{A}(\Omega_T)$.

The error estimate is simply obtained considering $(\b v_a + \varepsilon \b w)_{|x_d>0} - \b v =  \varepsilon \b V^{1,+} + \dots + \varepsilon^M \b V^{M,+} + \varepsilon \b w$ in $]-T_0,T[\times \mathbb{R}^d_+$.

This finishes the proof of Theorem \ref{Th_reform}.

\section{A first example}\label{sect_example_lin}
This section contains a simple application of the main result for a one-dimensional and linear hyperbolic problem. The fact that the system is linear enables us to compare with others penalty methods such as \cite{For09, Rau79} obtained for this case.

For this simple example, $\bar{\b A}$ is a constant symmetric matrix of size $N\times N$ and $\b C$ is a constant matrix of size $p \times N$ whose the rank is $p \leq N$. 
\begin{equation*}\label{lin_system}
\displaystyle
\left\{\begin{array}{ll} 
\partial_t \b u(t, x) + \bar{\b A} \partial_x \b u(t, x) = \bar{\b f}(t, x) & (t, x) \in ]-T_0,T[ \times \mathbb{R}^1_+ \\
\b C \b u(t,0)=\b 0 & t \in ]-T_0,T[ \\
\b u_{|t<0} = \b 0 & 
\end{array}
\right.
\end{equation*}
Assume that all hypothesis of the section \ref{sect_main_result} are satisfied.
Besides, the submatrix $\b C_{p\times p}$ composed of the $p$ first columns of $\b C$ is supposed to be invertible.

The first step is the change of unknown. For the change of unknown of the lemma \ref{lemma_ch_unknown}, we choose:
\begin{equation*}
 \b v=\left(\begin{array}{c} \b C_{p\times p} \left(\begin{array}{c} u_1\\ \vdots \\ u_p \end{array}\right) \\ u_{p+1}\\ \vdots \\ u_N \end{array}\right)
\end{equation*}

In this case, the change of unknown $\b H$ and its gradient are the following linear maps:
\begin{equation*}
\b H : \b v \mapsto \left(\begin{array}{c} \b C_{p\times p}^{-1} \left(\begin{array}{c} v_1\\ \vdots \\ v_p \end{array}\right) \\ v_{p+1}\\ \vdots \\ v_N \end{array}\right)
\text{ and }
\nabla_{\b v} \b H(\b v)=\left(\begin{array}{c|c}
\b C_{p\times p}^{-1} 	&	\b 0				\\
\hline
				\b 0	&	\b {I}_{N-p}	\\			
\end{array}\right)
\end{equation*}
where $\b {I}_{N-p}$ is the identity matrix of $\mathbb{R}^{N-p}$.

Finally the penalty matrix is
\begin{equation*}
\b M =
\left(\begin{array}{c|c}
\b C_{p\times p}^\top 	&	\b 0			\\
\hline
				\b 0	&	\b{I}_{N-p}	\\			
\end{array}\right)
\left(\begin{array}{c|c}
\b{I}_p 	&	\b 0	\\
\hline
	\b 0	&	\b 0	\\		
\end{array}\right)
\left(\begin{array}{c|c}
\b C_{p\times p} 	&	\b 0			\\
\hline
	 			\b 0	&	\b{I}_{N-p}	\\			
\end{array}\right)
=\left(\begin{array}{c|c}
\b C_{p\times p}^\top \b C_{p\times p} 	&	\b 0			\\
\hline
				\b 0	&	\b 0	\\			
\end{array}\right)
\end{equation*}
and the penalized problem writes, in the original unknowns:
\begin{equation*}\label{lin_system_pen}
\displaystyle
\left\{\begin{array}{ll} 
\partial_t \b u_\varepsilon + \bar{\b A} \partial_x \b u_\varepsilon + \frac1\varepsilon \b M \b u_\varepsilon = \bar{\b f} & \text{ in } ]-T_0,T[ \times \mathbb{R} \\
\b u_{\varepsilon \, |t<0} = \b 0 & 
\end{array}
\right.
\end{equation*}

In the results of Rauch \cite{Rau79}, the generation of the penalty matrix needs to find a positive definite matrix $\b E$ such that $\ker \b C$ is the subspace of the eigenvectors associated to the negative or null eigenvalues of $\b E \bar{\b A}$. In this case, the penalization matrix is $\b \Psi^{\top} \b \Psi$ where $\b \Psi=\b O \b E^{\frac12}$ and $\b O$ represents any orthogonal matrix. Theorem 2.7 of the paper of Fornet and Gu\`es \cite{For09} proposes a penalization matrix of the form $\left(\b \Psi^{-1}\right)^{\top} \mathbb{P} \b \Psi^{-1}$ where $\mathbb{P}$ is the projector of $\mathbb{R}^N$ onto $\b \Psi^{-1} \ker \b C$. Finally, for the use of this two penalty methods (Rauch and Fornet-Gu\`es), the more difficult point is to find a suited matrix $\b E$ and to compute $\b E^{\frac12}$.

Our method is more direct, even in the linear case, and has been extended to the quasilinear case.
The main difficulty is the change of unknown, which is provided by the proof of the lemma \ref{lemma_ch_unknown}. Moreover, for this example, the expression of the penalty matrix is simple.

\section{An example of application in plasma physics}\label{Example_appl}
This section shows quickly how the penalty method presented in this paper can be applied for the numerical simulation of the edge plasma transport, for more details see \cite{Ang12}. In the toy model presented below, the first equation stands for the mass conservation and the second one for the momentum conservation. $N$ represents the plasma density, $\Gamma$ the plasma momentum and $M=\Gamma/N$ the Mach number. The space variable $x$ stands for the curvilinear coordinate along a magnetic field line.
\begin{equation*}
 \left\{\begin{array}{l}
 \partial_t N + \partial_x \Gamma = S_N\\
 \partial_t \Gamma + \partial_x\left(\dfrac{\Gamma^2}{N} + N\right)=S_{\Gamma}\\
 \left( \begin{array}{cc} M_0 & -1 \end{array} \right) \left( \begin{array}{c} N(t,0) \\ \Gamma(t,0) \end{array} \right) = 0
 \end{array}\right.
\end{equation*}
Where $S_N$ and $S_\Gamma$ are source terms of the hyperbolic problem.

Notice that the system is very similar to shallow water equations.
The change of variable used to reformulate the system is:
\begin{align*}
 \tilde{u}(t, x)=\ln \left( N(t, x) \right)\\
 \tilde{v}(t, x)=\dfrac{\Gamma(t, x)}{N(t, x)} - M_0
\end{align*}
Hence, only $\tilde{v}$ is affected by the boundary condition. 

Finally, the penalization obtained thanks to the results presented above is:
\begin{equation*}
 \left\{\begin{array}{l}
 \partial_t N + \partial_x \Gamma = S_N\\
 \partial_t \Gamma + \partial_x\left(\dfrac{\Gamma^2}{N} + N\right) + \frac{\chi}{\varepsilon}\left(\dfrac{\Gamma}{M_0}-N\right)=S_{\Gamma}\\
 \end{array}\right.
\end{equation*}

The main advantage of this method is the absence of spurious boundary layer: the error due to the penalization decreases with an optimal rate when the penalization parameter $\varepsilon$ tends to $0$.
The main drawback is due to the fact that the penalization is incomplete. Thus, at the boundary of the computational domain, we need to provide transparent boundary condition, at least for the non-penalized field $\tilde{u}$, which is not easy. Besides, non compatible initial boundary condition may generates artefact, see for intance, the numerical results of \cite{Ang12}.

\section{Conclusion}
This paper provides a penalty method to take into account of the boundary conditions of a non characteristic quasilinear hyperbolic problem which is in fact quite natural: after a change of unknown, one penalizes only the fields concerned by the boundary condition. An interesting feature of this method, is that the error due to the penalization has an optimal rate of convergence, \emph{i.e.} $\|\b u - \b u_{\varepsilon}\|_{H^m}=\mathcal{O}(\varepsilon)$. To focus our work on the penalization, we consider regular functions and solution null in the past to avoid initial condition compatibility issues.

This method has already been tested numerically in a one-dimensional non linear hyperbolic problem.

For further works, it might be interesting to extend this results to characteristic problems, such as in \cite{For09}.

\subsection{Acknowledgements}
This work has been funded by the ANR ESPOIR (Edge Simulation of the Physics Of ITER Relevant turbulent transport) and the \emph{F\'ed\'eration nationale de Recherche sur la Fusion par Confinement Magn\'etique} (FR-FCM). We thank Philippe Angot and Olivier Gu\`es for fruitful discussions and assistance.

\bibliographystyle{plain}
\bibliography{biblio_article}

\section*{Appendix: Some recalls about notations}
\begin{tabular}{|p{0.17\textwidth}|p{0.75\textwidth}|}
\hline
Object							& Definition or explanations \\
\hline
$\nabla_{\b u}, \nabla_{\b v}$	& Partial derivatives relative to the vector $\b u$ or $\b v$ (respectively).\\
$\nabla$						& Gradient relative to the variables $t, \b x$.\\
$\b a$							& A function representing the dependence on $(t,\b x)$ of the coefficients of the hyperbolic problem.\\
$\b A_j(\varepsilon \b b)$ 		& $\begin{array}{l}\b A_j(\varepsilon \b b)=\b A_j(\b v_a + \varepsilon \b b)=\b A_j(. , . , \b v_a + \varepsilon \b b)\\
\quad = \left(\nabla_{\b v} \b H(. , . ,\!\b v_a \! + \! \varepsilon \b b)\right)^{-1}\! \b S\left(. , . ,\b H(. , . ,\!\b v_a \! + \! \varepsilon \b b)\right) \tilde{\b A}_j(. , . ,\!\b v_a \! + \!\varepsilon \b b) \nabla_{\b v} \b H(. , . ,\!\b v_a \! + \! \varepsilon \b b) \end{array}$\\
$\mathcal{A}(\Omega_T)$ 		& Functional space, see definition \ref{A-space}.\\
$cst, cst(R)$					& Constant which does not depend en $\tilde{\b w}, \b b, \varepsilon, \lambda$.\\
$H^m_{tan}(\Omega_T)$			& $H^m_{tan}(\Omega_T)= \left\{\b \Phi \in L^2(\Omega_T), \forall \alpha \in \mathbb{N}^d, |\alpha|\leq m \Longrightarrow \mathcal{T}^{\alpha} \b \Phi \in L^2(\Omega_T) \right\}$\\
$H^m(\Omega_T)$					& $H^m(\Omega_T)= \left\{\b \Phi \in L^2(\Omega_T), \forall \beta \in \mathbb{N}, \forall \alpha \in \mathbb{N}^d,  \beta+|\alpha|\leq m \Longrightarrow \partial_d^{\beta} \mathcal{T}^{\alpha} \b \Phi \in L^2(\Omega_T) \right\}$\\
$L^2(\Omega_T)$					& $L^2(\Omega_T) = \{\b \Phi : \Omega_T \to \mathbb{R}^N, \int_{\Omega_T}{ \langle \b \Phi(t,\b x) , \b \Phi(t,\b x) \rangle dt d\b x}\}$\\
$\mathcal{M}_N(\mathbb{R})$		& The set of square matrix of size $N \times N$.\\
$\b P$							& Projection matrix of rank $p$, $\b P \b v_{x_d=0}=\b 0$.\\
$t$								& Time variable.\\
$\mathcal{T}$					& Tangential derivatives.\\
$\b u$ 							& Solution of the initial hyperbolic problem.\\
$\b u_{\varepsilon}$			& Solution of the penalized hyperbolic problem.\\
$\b U, \b V, \b W$				& Any element of $\mathbb{R}^{N}$, eventually in a chosen neighbourhood of $\b 0$.\\
$\b v$							& Solution of the hyperbolic problem with the new unknown (to have the boundary condition $\b P \b v = \b 0$).\\
$\b v_a$						& First terms of the asymptotic expansion: $ \b v_a = \sum_{n=0}^M{\b V^{n \pm}}$\\
$\b v_{\varepsilon}$			& Solution of the penalized hyperbolic problem with the new unknown.\\
$\b w$							& $\b w = \frac{1}{\varepsilon}\left(\b v - \b v_a \right)$\\
$W^{1,\infty}(\Omega_T)$		& $W^{1,\infty}(\Omega_T) = \left\{\b \Phi \in L^{\infty}(\Omega_T), \nabla \b \Phi \in L^{\infty}(\Omega_T) \right\}$\\
$\b x, \b x'$					& $\b x=(x_1,\dots,x_d)=(\b x',x_d) \in \mathbb{R}^{d}$ space variable.\\
$\b y$							& Any element of $\mathbb{R}^{N'}$.\\
								& \\
$\varepsilon$					& Penalization parameter (normally next to $0$).\\ 
$\b \Theta(.,., \dots,., \b u) \! = \! \b 0$ & Boundary condition for the original hyperbolic boundary value problem.\\
$\Omega_T$ 						&$\Omega_T=]-T_0,T[ \times \mathbb{R}^d$\\
$\Omega_T^+$ 					&$\Omega_T^+=]-T_0,T[ \times \mathbb{R}^d_+$\\
$\Omega_T^-$ 					&$\Omega_T^-=]-T_0,T[ \times \mathbb{R}^d_-$\\
								& \\
$^{\top}$						& Matrix transposition.\\
$\partial_0=\partial_t$			& Time derivative.\\
$\partial_j$					& $\partial_j=\partial_{x_j}=\frac{\partial}{\partial x_j}$\\
$\mathbb{R}^d_+$				& $\{(x_1,\dots,x_d) \in \mathbb{R}^d, x_d>0\}$\\
$\mathbb{R}^d_-$				& $\{(x_1,\dots,x_d) \in \mathbb{R}^d, x_d<0\}$\\
$\langle . , . \rangle$			& Euclidean scalar product on $\mathbb{R}^N$.\\
$\langle . , . \rangle_{\mathbb{R}^{2 N}}$			& Euclidean scalar product on $\mathbb{R}^{2 N}$.\\
$\|.\|$							& Euclidean norm on $\mathbb{R}^N$.\\
$\langle . , . \rangle_{L^2(\Omega_T)}$ & Usual inner product on $L^2(\Omega_T)$.\\
\hline
\end{tabular}

\end{document}